\documentclass[a4paper, 12pt]{amsart}

\usepackage{verbatim}                          
\usepackage{amsmath,amsthm}                                        
\usepackage{color,textcomp,amssymb,german}
\usepackage{graphicx}
\usepackage[british]{babel}
\usepackage[all,import]{xy}\CompileMatrices\CompilePrefix{xymatrix/diagram} 
\usepackage[T1]{fontenc}                                    
\usepackage{times}                                       
\usepackage[mathscr]{eucal}

\swapnumbers
\newtheoremstyle{theorem}{8pt\vfill}{8pt\vfill}{\itshape}{}{\bfseries}{.}{.5em}{}
\newtheoremstyle{paragraph}{8pt\vfill}{8pt\vfill}{}{}{\bfseries}{}{.5em}{}
\theoremstyle{theorem}                                       
\newtheorem{thm}{Theorem}[section]
\newtheorem{conj}[thm]{Conjecture}
\newtheorem{prop}[thm]{Proposition}
\newtheorem{cor}[thm]{Corollary}
\newtheorem{lemma}[thm]{Lemma}
\theoremstyle{definition}
\newtheorem{df}[thm]{Definition}
\newtheorem{rem}[thm]{Remark}
\theoremstyle{paragraph}                                     
\newtheorem{pg}[thm]{\!\!}
 
\newcommand{\norm}[1]{\left| #1 \right|}
\newcommand{\lrquot}[3]{#1 \,\backslash\, #2 \,/\, #3 }
\newcommand{\lquot}[2]{#1  \,\backslash\, #2 }
\newcommand{\rquot}[2]{#1  \, / \, #2 }

\def \mat #1 #2 #3 #4 { \begin{pmatrix} {#1} & #2 \\ #3 & #4 \end{pmatrix} }
\def \smallmat #1 #2 #3 #4 {{\scriptstyle \begin{pmatrix} {#1} & {#2} \\ {#3} & {#4} \end{pmatrix}}}
\def \tinymat #1 #2 #3 #4 {\bigl( \begin{smallmatrix} {#1} & {#2} \\ {#3} & {#4} \end{smallmatrix} \bigr)}
\def\eqref #1{\textup{(}\ref{#1}\textup{)}}
\def\cprime{$'$}

\def\blanc{\mspace{4mu}\cdot\mspace{4mu}}
\def\cO{\mathcal{O}}
\def\cP{\mathcal{P}}
\def\cA{\mathcal{A}}
\def\cH{\mathcal{H}}
\def\cE{\mathcal{E}}
\def\cR{\mathcal{R}}

\def\ZZ{\mathbb{Z}}
\def\QQ{\mathbb{Q}}
\def\RR{\mathbb{R}}
\def\CC{\mathbb{C}}
\def\FF{\mathbb{F}}
\def\AA{\mathbb{A}}
\def\PP{\mathbb{P}}
\def\SS{\mathbb{S}}
\def\HH{\mathbb{H}}
\def\fc{\mathfrak{c}}
\def\fL{\mathfrak{L}}
\def\N{\mathrm{N}}
\def\Z{\mathrm{Z}}
\def\varE{\tilde E}
\def\varcE{\widetilde\cE}
\DeclareMathOperator{\Hom}{Hom}
\DeclareMathOperator{\Mat}{Mat}
\DeclareMathOperator{\im}{im}
\DeclareMathOperator{\Gal}{Gal}
\DeclareMathOperator{\Res}{Res}
\DeclareMathOperator{\ch}{char}
\DeclareMathOperator{\GL}{GL}
\DeclareMathOperator{\PSL}{PSL}
\DeclareMathOperator{\ord}{ord}
\DeclareMathOperator{\vol}{vol}
\DeclareMathOperator{\pr}{pr}
\def\tor{{\rm tor}}

\def\nr{{\rm nr}}
\def\Re{\mathop{\rm Re}\nolimits}
\def\Im{\mathop{\rm Im}\nolimits}
\def\i{{\rm i}}

\renewcommand{\date}[1]{\gdef\DD{#1}}\newcommand{\DD}{}

\title{Toroidal automorphic forms for function fields}
\author{Oliver Lorscheid}
\address{The City College of New York, Math. Dept., 160 Convent Ave., New York NY 10031, USA}
\email{olorscheid@ccny.cuny.edu}

\begin{document}

\begin{abstract}
 The space of toroidal automorphic forms was introduced by Zagier in 1979. Let $F$ be a global field. An automorphic form on $\GL(2)$ is toroidal if it has vanishing constant Fourier coefficients along all embedded non-split tori. The interest in this space stems from the fact (amongst others) that an Eisenstein series of weight $s$ is toroidal if $s$ is a non-trivial zero of the zeta function, and thus a connection with the Riemann hypothesis is established.

 In this paper, we concentrate on the function field case. We show the following results. The $(n-1)$-th derivative of a non-trivial Eisenstein series of weight $s$ and Hecke character $\chi$ is toroidal if and only if $L(\chi,s+1/2)$ vanishes in $s$ to order at least $n$ (for the ``only if''-part we assume that the characteristic of $F$ is odd). There are no non-trivial toroidal residues of Eisenstein series. The dimension of the space of derivatives of unramified Eisenstein series equals $h(g-1)+1$ if the characterisitc is not $2$; in characteristic $2$, the dimension is bounded from below by this number. Here $g$ is the genus and $h$ is the class number of $F$. The space of toroidal automorphic forms is an admissible representation and every irreducible subquotient is tempered. 

\end{abstract}


\maketitle

\tableofcontents


\section*{Introduction}

\noindent
At the Bombay Colloquium in January 1979, Don Zagier (\cite{Zagier1}) observed that if the kernel of certain operators on automorphic forms turns out to be a unitarizable representation, a formula of Hecke implies the Riemann hypothesis. 

We review this idea in classical language. Let $\Gamma=\PSL_2\ZZ$ and $\Gamma_\infty=\{\pm\tinymat 1 n 0 1 | n\in\ZZ\}$. These groups both act on the Poincar\'e upper half plane $\HH=\{z\in\CC|\Im z>0\}$ by M\"obius transformations. For $z\in\HH$ and $s\in\CC$, let
$$ E(z,s) \quad = \quad \pi^{-s}\Gamma(s)\zeta(2s)\sum_{\gamma\in\lquot{\Gamma_\infty}{\Gamma}} \Im(\gamma z)^s $$
be the complete Eisenstein series where $\zeta(s)=\sum_{n\geq1}n^{-s}$ is the Riemann zeta function. We note that all formulas only make sense for $\Re s>1$, but admit a meromorphic continuation to all $s\in\CC$. Let $E=\QQ[\sqrt D]$ be an imaginary quadratic number field of discrimant $D<0$. To each positive binary quadratic form $Q(m,n)=am^2+bmn+cn^2$ of discriminant $b^2-4ac=D$, we associate the root $z_Q=\frac{-b+\sqrt D}{2a}\in\HH$. Let $\{z_1,\dotsc,z_r\}$ be the inequivalent points in $\HH$ $(\textup{mod }\Gamma)$ that are of the form $z_Q$ for a positive binary quadratic form $Q$ of discriminant $D$. The following formula is  due to Hecke (\cite[p.\ 201]{Hecke}):
$$ \sum_{i=1}^r E(z_i,s) \quad = \quad \frac w2\ \norm D^{s/2}\ (2\pi)^{-s}\ \Gamma(s)\ \zeta_E(s) $$
where $w$ is the number of the roots of unity contained in $E$ and $\zeta_E(s)$ is the Dedekind zeta function of $E$. Since $\zeta_E(s)=\zeta(s)L(\chi_E,s)$ where $\chi_E$ is the quadratic Hecke character associated to $E$ by class field theory, we observe that $\sum_{i=1}^r E(z_i,s)$ vanishes if $s$ is a zero of $\zeta(s)$; more generally, $\sum_{i=1}^r \frac{d^{(n-1)}}{ds^{(n-1)}}E(z_i,s)=0$ if $s$ is a zero of $\zeta(s)$ of order at least $n$. 

Zagier's observation in \cite{Zagier1} is that an Eisenstein series $E(z,s)$ lies in a unitarizable automorphic representation of $\PSL_2\RR$ if and only if $s\in(0,1)$ or $\Re s=1/2$. Proper derivatives $\frac{d^{(i)}}{ds^{(i)}}E(z,s)$ of Eisenstein series do not lie in a unitarizable representation. Note that $\zeta(s)$ has no zero in $(0,1)$ (cf.\ \cite[Formula (2.12.4)]{Titchmarsh}). Thus if the space of all $\Gamma$-invariant automorphic forms $f$ on $\HH$ that satisfy $\sum_{i=1}^rf(z_i)=0$ is a unitarizable automorphic representation, then the Riemann hypothesis follows and $\zeta(s)$ has no multiple zeros.

The adelic analogue of Hecke's formula can be formulated as follows. Let $G=\GL(2)$ and $Z$ its center. Let $\AA$ be the adeles of $\QQ$. A choice of a $\QQ$-basis for $E$ defines a non-split torus $T(F)=\im(E^\times\hookrightarrow G(\QQ))$ of $G$. Note that there is a projection $G(\AA)\to\lquot{\Gamma}{\HH}$ such that $T(\AA)$ maps precisely to the points $\{z_1,\dotsc,z_r\}$. Let $E(g,s)$ be the spherical unramified Eisenstein series of weight $s$ (note that the weight in the adelic language is usually shifted by $1/2$ compared to classical Eisenstein series). Then the adelic version of Hecke's formula reads as
$$ \int_{\lquot{T(F)Z(\AA)}{T(\AA)}} E(tg,\varphi,s)\;dt \quad = \quad c(g,\varphi,s)\ \Gamma(s+1/2)\ \zeta_E(s+1/2) $$
for some factor $c(g,\varphi,s)$ which is holomorphic in $s$. 

This formula generalises to all global fields $F$ and all quadratic field extensions $E$ of $F$. We are concerned with the case that $F$ is a function field in this paper. Following Zagier, we define the \emph{space $\cA_\tor(E)$ of $E$-toroidal automorphic forms} as the space of automorphic forms $f$ for which the toroidal integrals
$$\int_{\lquot{T(F)Z(\AA)}{T(\AA)}} f(tg)dt $$ 
vanish for all $g\in G(\AA)$. The \emph{space of toroidal automorphic forms}\footnote{This definition differs slightly from the definition in the main text since we will also make sense of toroidal integrals with respect to split tori, which will contribute to the definition of $\cA_\tor$. A posteriori, however, it will follow for odd characteristic that these two definitions coincide (cf.\ Remark \ref{rem_omit_E}).} is $\cA_\tor=\bigcap\cA_\tor(E)$ where $E$ varies through all separable field extensions of $F$. Zagier raises questions at the end of \cite{Zagier1} for the function field case: Is the spectrum of $\cA_\tor$ discrete? What is the dimension of the space of unramified automorphic forms in $\cA_\tor$? In particular, is it finite-dimensional? 

In this paper we give answers to these questions (and more). In particular, we prove the following statements.
\begin{enumerate}
 \item Let the characteristic $p$ be odd. Then the $(n-1)$-st derivatives of all Eisenstein series of weight $s$ and Hecke character $\chi$ are toroidal if and only if $L(\chi,s+1/2)$ vanishes in $s$ to order at least $n$ (Theorem \ref{toroidal_if_and_only_L-series=0}). The ``if''-part holds also for $p=2$ (Theorem \ref{lemma_zagier_for_derivatives}).
 \item There are no non-trivial toroidal residues of Eisenstein series (Theroem \ref{no_toroidal_residues}).
 \item The space $\cA_\tor$ is an admissible automorphic representation (Theorem \ref{thm_admissible}). In particular, the spectrum of $\cA_\tor$ is discrete and the space of unramified $E$-toroidal automorphic forms is finite-dimensional.
 \item Let $g$ be the genus of $F$ and $h$ the class number. If $p$ is odd, then the dimension of the space of unramified toroidal derivatives of Eisenstein series is $h(g-1)+1$. If $p=2$, then $h(g-1)+1$ is a lower bound for the dimension (Theorem \ref{thm_dimension}).
 \item Every irreducible subquotient of $\cA_\tor$ is a tempered automorphic representation (Theorem \ref{thm_tempered}). In particular, every irreducible subquotient of $\cA_\tor$ is unitarizable.
\end{enumerate}

We briefly review the developments after the appearance of Zagier's paper \cite{Zagier1}. The work of Waldspurger on the Shimura correspondence (\cite{Waldspurger4}, \cite{Waldspurger1}, \cite{Waldspurger2} and \cite{Waldspurger3}) includes a formula connecting toroidal integrals of cusp forms (nowadays also called Waldspurger periods) with the value of the $L$-function of the corresponding cuspidal representation at $1/2$. In \cite{Wielonsky2} Franck Wielonsky worked out Zagier's ideas and obtained a generalisation to a limited class of Eisenstein series on $\GL_n(\AA)$. Lachaud tied up the spaces with Connes' view on the zeta function, cf.\ \cite{Lachaud1} and \cite{Lachaud2}. Clozel and Ullmo (\cite{Clozel-Ullmo}) used both Waldspurger's and Zagier's works to prove a equidistribution result for tori in $\GL_2$, and Lysenko (\cite{Lysenko}) translated certain Waldspurger periods into geometric language. In a joint work with Cornelissen (\cite{Cornelissen-Lorscheid}), we calculated the space of unramified toroidal automorphic forms for global function fields with a rational point of genus $g\leq1$ and class number $1$. In another joint paper with Cornelissen (\cite{Cornelissen-Lorscheid2}), we describe the space of toroidal automorphic forms in the number field case.

The paper is divided in four parts, each of which contains different sections. In Part \ref{part_notation}, we give notations and definitions. In Section \ref{section_automorphic}, we introduce automorphic forms for $\GL(2)$, the Hecke algebra and cusp forms. In Section \ref{section_toroidal}, we define $E$-toroidal automorphic forms for every separable quadratic algebra extension $E$ of $F$. In particular, we include a definition for the split torus, which corresponds to the algebra $E=F\oplus F$.

In Part \ref{part_eisenstein}, we draw conclusions about the space of toroidal Eisenstein series by various methods. In Section \ref{section_review}, we review the definitions of and some results about $L$-series and Eisenstein series. In Section \ref{section_non-split}, we review the adelic version of Hecke's formula in detail and give, in particular, non-vanishing results for the factor appearing in the equation. This yields a precise description of the space of $E$-toroidal Eisenstein series. In Section \ref{section_split}, we establish the corresponding formula for split tori, which is analogous to the case of the non-split torus, though, proven differently. In Section \ref{section_derivatives}, we draw conclusions about toroidality of derivatives of Eisenstein series. In Section \ref{section_residues}, we determine which residues of Eisenstein series are $E$-toroidal and show that there are no non-trivial toroidal residues of Eisenstein series. In Section \ref{section_non-vanishing} we employ Double Dirichlet series to show that in odd characteristic, the quadratic twists $L(\chi\chi_E,s)$ do not vanish simultaneously for a given Hecke character $\chi$ and a given $s\in\CC$ when $\chi_E$ varies through all non-trivial quadratic Hecke characters. This yields a precise description of toroidal derivatives of Eisenstein series.

In Part \ref{part_representations}, we consider the properties of $\cA_\tor$ as a representation. In Section \ref{section_riemann}, we formulate the implication of temperedness on the Riemann hypothesis for $F$. In particular, we formulate a sufficient condition on the eigenvalue of a single Hecke operator on unramified toroidal Hecke-eigenforms, which can be verified in examples (see \cite{Cornelissen-Lorscheid} and \cite{Lorscheid-thesis}). In Section \ref{section_tempered}, we show that every irreducible subquotient of $\cA_\tor$ is a tempered automorphic representation and that $\cA_\tor$ is admissible.

In Part \ref{part_dimension}, we establish a dimension formula for the space of derivatives of unramified Eisenstein series. In Section \ref{section_basis}, we establish a basis for the space of unramified automorphic forms, which consints in generalised Hecke eigenforms. In particular, we investigate all linear dependencies between derivatives of (residual) Eisenstein series. In Section \ref{section_dimension}, we use this basis to show that the dimension of the space of derivatives of unramified Eisenstein series equals $h(g-1)+1$ (resp.\ is bounded by, in characteristic $2$) where $g$ is the genus and $h$ is the class number of $F$.

\medskip
\noindent{\bf Acknowledgements: }
This paper as well as \cite{Lorscheid3} and \cite{Lorscheid2} are extracted from my thesis \cite{Lorscheid-thesis} (except for Section \ref{section_non-vanishing}, which follows \cite{Cornelissen-Lorscheid2}). First of all, I would like to thank my thesis advisor Gunther Cornelissen who guided and helped me in in my studies on toroidal automorphic forms. I would like to thank Don Zagier and G\"unter Harder for explaining to me their ideas on the topic. I would like to thank Roelof Bruggeman and Frits Beukers for their comments on many lectures and drafts that formed the blueprint of my thesis. I would like to thank Gerard Laumon, Laurent Clozel and Jean-Loup Waldspurger for their hospitality and mathematical help during a fruitful month in Paris.

\bigskip

\part{Notations and Definitions}
\label{part_notation}

\noindent
This first part introduces the notation used throughout the paper and sets up the definition of the space of toroidal automorphic forms.


\section{Automorphic forms for $\GL(2)$}
\label{section_automorphic}

\begin{pg}
 \label{def_K}
 Let $q$ be a prime power and $F$ be a global function field with constants $\FF_q$. Let $X$ be the set of all places of $F$. We denote by $F_x$ the completion of $F$ at $x\in X$ and by $\cO_x$ the integers of $F_x$. We choose a uniformizer $\pi_x\in F$ for every place $x$. Let $\deg x$ be the degree of $x$ and let $q_x=q^{\deg x}$ be the cardinality of the residue field of $\cO_x$. We denote by $\norm \ _x$ the absolute value on $F_x$ resp.\ $F$ such that $\norm{\pi_x}_x=q_x^{-1}$.

 Let $\AA$ be the adele ring of $F$ and $\AA^\times$ the idele group. Put $\cO_\AA=\prod\cO_x$ where the product is taken over all places $x$ of $F$. Let $g$ be the genus of $F$. Let $\fc$ be a differental idele, i.e.\ a representative of the canonical divisor in the divisor class group $\AA^\times/F^\times\cO_\AA^\times$, which is of degree $2g-2$. The idele norm is the quasi-character $\norm\ :\AA^\times\to\CC^\times$ that sends an idele $(a_x)\in\AA^\times$ to the product $\prod\norm{a_x}_x$ over all local norms. By the product formula, this defines a quasi-character on the idele class group $\rquot{\AA^\times}{F^\times}$.

 Let $\Xi$ be the group of all quasi-characters on the the idele class group, i.e.\ the group of all continuous group homomorphisms $\chi:\rquot{\AA^\times}{F^\times}\to\CC$. Let $\Xi_0$ be the subgroup of unramified quasi-characters, i.e.\ the group of those quasi-characters that are $\cO_\AA$-invariant. Note that every quasi-character $\chi$ is of the form $\chi=\chi_0\norm\ ^s$ for a complex number $s$ modulo $2\pi\i/\ln q$ and a character $\chi_0$ of finite order; in particular $\im\chi_0\subset\SS^1$. Though there are different choices for such a decomposition into a finite character and a principal quasi-character, the real part of $s$ is independent of the decomposition and we define $\Re\chi$ as the real part of $s$.

 Let $G=\GL(2)$ be the algebraic group of invertible $2\times 2$-matrices and let $Z$ be the center of $G$. Following the habit of literature about automorphic forms, we will often write $G_\AA$ instead of $G(\AA)$ for the group of adelic points and $Z_F$ instead of $Z(F)$ for the group of $F$-valued points, et cetera. Let $K_x=G(\cO_x)$ and $K=\prod K_x= G(\cO_\AA)$, which is a maximal compact subgroup of $G_\AA$. The adelic topology turns $G_\AA$ into a locally compact topological group with maximal compact subgroup $K$.
\end{pg}

\begin{pg}
 Let $\cH$ be the \emph{Hecke algebra for $G_\AA$}, which is the vector space of all compactly supported locally constant functions $\Phi: G_\AA\to\CC$ together with the convolution product
 $$ \Phi_1 \ast \Phi_2: g \mapsto \int\limits_{G_\AA} \Phi_1(gh^{-1})\Phi_2(h) \,dh\,, $$
 A Hecke operator $\Phi\in\cH$ acts on the space $C^0(G_\AA)$ of continuous functions $f:G_\AA\to\CC$ by the formula
 $$ \Phi.f:g\to\int\limits_{G_\AA}\Phi(h)f(gh)\,dh. $$

 A function $f\in C^0(G_\AA)$ is called \emph{smooth} if it is locally constant and it is called \emph{$K$-finite} if the set of all right $K$-translates $f_k:g\to f(gk)$ generates a finite-dimensional subspace of $C^0(G_\AA)$.

 An \emph{automorphic form on $G_\AA$ (with trivial central character)} is a smooth, $K$-finite and left $G_FZ_\AA$-invariant function $f:G_\AA\to\CC$ such that for every Hecke operator $\Phi\in\cH$, the functions $\Phi^i.f$ for $i\geq 0$ generate a finite-dimensional subspace of $C^0(G_\AA)$. We denote the space of all automorphic forms by $\cA$. The group $G_\AA$ acts on $\cA$ via the \emph{right-regular representation}, i.e.\ $g.f:h\to f(hg)$. The action of $\cH$ restricts to the subspace $\cA$ of $C^0(G_\AA)$.

 There is a one-to-one correspondence between $G_\AA$-modules and $\cH$-modules. In particular, a subspace of $\cA$ is a $G_\AA$-module if and only if it is an $\cH$-module. In the following, we will often speak of subrepresentations of $\cA$ without specifying $G_\AA$ or $\cH$ explicitly.


\end{pg}

\begin{pg}
 \label{cusp_forms}
 Let $B$ be a Borel subgroup of $G$ and $N\subset B$ its unipotent radical. The \emph{constant term $f_N$ (with respect to $N$)} of an automorphic form $f\in\cA$ is the function $f_N:G_\AA\to \CC$ defined by
 $$ f_N(g) \ := \ \vol(\lquot{N_F}{N_\AA})^{-1}\int\limits_{\lquot{N_F}{N_\AA}} f(ng)\,dn. $$
 The constant term $f_N$  is left $B_FZ_\AA$-invariant. If $f_N(g)$ vanishes for all $g\in G_\AA$, the automorphic form $f$ is called a \emph{cusp form}. We denote the space of all cusp forms by $\cA_0$. If $e\in G_\AA$ denotes the identity matrix, then we have the alternative descripription
 $$ \cA_0 = \{f\in\cA\mid\forall\Phi\in\cH, \Phi(f)_N(e)=0\}. $$

 \label{approx_by_const_terms}
 The \emph{approximation by constant terms} (\cite[I.2.9]{Moeglin-Waldspurger}) states that for every $f\in\cA$, the function $f-f_N$ has compact support as a function on $\lquot{B_FZ_\AA}{G_\AA}$.
\end{pg}


\section{Toroidal automorphic forms}
\label{section_toroidal}

\begin{pg}
 \label{anisotropic_torus_is_compact}\label{bijection_quadr_ext_conj_cl_of_tori} 
 Let $T$ be a maximal torus of $G$. Then either $T$ is \emph{split}, i.e.\ $T(F)$ is isomorphic to the units of $F\oplus F$, or $T$ is \emph{non-split}, i.e.\ $T(F)$ is isomorphic to the units of a separable quadratic field extension $E$ of $F$. This establishes a bijection between the conjugacy classes of maximal tori in $G$ and the isomorphism classes of separable quadratic algebra extensions of $F$.

 If $T$ is an split torus, then $\lquot{T_FZ_\AA}{T_\AA}\simeq\lquot{F^\times}{\AA^\times}$. If $T$ is a non-split torus, i.e.\ $T(F)\simeq E^\times$ for a separable quadratic field extension $E/F$, then $\lquot{T_FZ_\AA}{T_\AA}\simeq \lquot{E^\times\AA^\times}{\AA_E^\times}$ is a compact abelian group. In this case, $T$ splits over $E$.
\end{pg}

\begin{pg}
 \label{def_toroidal_integral}
 Let $T$ be a non-split torus, and endow $Z_\AA$ and $T_\AA$ with Haar measures such that $\Z_\AA\simeq\AA_F^\times$ and
 $T_\AA\simeq\AA_E^\times$ as measure spaces. 
 Endow $T_F$ with the discrete measure. This defines a Haar measure on $\lquot{T_FZ_\AA}{T_\AA}$ as quotient measure. We call
 $$ f_T(g) \ := \ \int \limits_{\lquot{T_FZ_\AA}{T_\AA}} f(tg)\,dt $$
 the {\it toroidal integral of $T$ (evaluated at $g$)}. Since the domain is compact, the integral converges for all $f \in \cA$ and $g\in G_\AA$.

 If $T$ is a split torus, then endow $T_\AA\simeq\AA^\times\oplus\AA^\times$ with the product measure of $\AA^\times$. Further let $Z_\AA$ carry the same measure as before and let $T_F$ carry the discrete measure. This defines a quotient measure on $\lquot{T_FZ_\AA}{T_\AA}$. Let $B$ and $B'$ be the two Borel subgroups that contain $T$, and let $N$ and $N^T$, respectively, be their unipotent radicals. Note that $\lquot{T_FZ_\AA}{T_\AA}$ is not compact, but due to approximation by constant terms (cf.\ paragraph \ref{cusp_forms}) both $f-f_N$ and $f-f_{N^T}$ have compact support as functions on $\lquot{B_FZ_\AA}{G_\AA}$ and $\lquot{B'_FZ_\AA}{G_\AA}$, respectively.  The {\it toroidal integral of $T$ (evaluated in $g$)} is
 $$ f_T(g) \ := \ \int \limits_{\lquot{T_FZ_\AA}{T_\AA}} \Bigl(f-\frac 12(f_N+f_{N^T})\Bigr)\,(tg)\,dt\;,$$
 which converges for all $f \in \cA$ and any choice of Haar measure on $\lquot{T_FZ_\AA}{T_\AA}$.
\end{pg}

\begin{df}
 \label{def_toroidal_forms}
 Let $T$ be a maximal torus of $G$ corresponding to a separable quadratic algebra extension $E/F$. Then define
 $$ \cA_\tor(E) \ = \ \{ f\in\cA\mid\forall g\in G_\AA, \; f_T(g)=0\} \;, $$
 the space of {\it $E$-toroidal automorphic forms}, and
 $$ \cA_\tor \ = \ \bigcap_{\begin{subarray}{c}\text{separable quadratic}\\\text{algebra extensions }E/F\end{subarray}} \cA_\tor(E) \;, $$
 the space of {\it toroidal automorphic forms}.
\end{df}

\begin{pg} 
 \label{tor_def_well-defined}
 The spaces $\cA_\tor(E)$ do not depend on the choice of torus in the conjugacy class corresponding to $E$. Indeed, for a conjugate $T_\gamma=\gamma^{-1} T\gamma$ with $\gamma\in G_F$, we have $f_{T_\gamma}(g)=f_T(g_\gamma)$, where $g_\gamma=\gamma g$. Note that the definition is also independent of the choices of Haar measures.

 \label{toroidal_form_def_by_H}
 As in the case of cusp forms, the correspondence between $G(\AA)$- and $\cH$-modules yields the following alternative descriptions. For all $T$ and $E$ as above,
 \begin{eqnarray*} \cA_\tor(E) & = & \{f\in\cA\mid\forall\Phi\in\cH, \Phi(f)_T(e)=0\} \qquad \text{and} \\
  \cA_\tor & = & \{ f\in\cA\mid\forall \text{ maximal tori }T\subset G,\,\forall \Phi\in\cH, \; \Phi(f)_T(e)=0\} \;. \end{eqnarray*}
\end{pg}

\bigskip

\part{Toroidal Eisenstein series}
\label{part_eisenstein}

\noindent
In this part of the paper we review Zagier's translation of Hecke's formula, which connects a sum of Eisenstein series to an $L$-series, into adelic language. Additionally we show non-vanishing results for the factors occuring in the formula. The case of split tori was not treated in adelic language yet and the proof is somewhat different to the case of the non-split torus; the result, however, is analogue to the non-split case. We begin with overviews of $L$-series and Eisenstein series to provide the reader with the results used in the latter.


\section{Review of $L$-series and Eisenstein series}
\label{section_review}

\begin{pg} 
 \label{cont_L_series}
 We give the necessary background on $L$-series that is used throughout the paper. For $\chi\in\Xi$ and $S=\{x\in X\mid\exists a_x\in\cO_x^\times,\;\chi(a_x)\neq1\}$, let
 $$ L_F(\chi,s)=\prod_{x\in X-S}\frac{1}{1-\chi(\pi_x)\norm{\pi_x}^s} $$
 be the \emph{$L$-series of $\chi$ in $s$}. If no confusion arises, we omit the subscript $F$ and write $L(\chi,s)$. This series converges for $\Re s>1-\Re\chi$ and it has a meromorphic continuation to all $s\in\CC$. It has poles in those $s$ for which $\chi\norm \ ^s$ equals $\norm\ ^0$ or $\norm \ ^1$ (i.e.\ $L(\chi,s)$ has poles only if $\chi$ has to be a principal character), and these poles are of order $1$. It satisfies a functional equation
 $$ L(\chi,1/2+s) = \epsilon(\chi,s) L(\chi^{-1},1/2-s) $$
 for a certain non-zero factor $\epsilon(\chi,s)$. If $\chi\in\Xi_0$, then $\epsilon(\chi,s)=\chi(\fc)\norm \fc^s$ where $\fc$ is a differental idele.
\end{pg}

\begin{pg}
 \label{Tate_integral}
 \label{Euler_prod_Tate_integral}
 We review the result known as Tate's thesis (cf.\ \cite[VII, Thm.\ 2 and \S\S\,6-7]{Weil}). Let $\psi:\AA\to\CC$ be a Schwartz-Bruhat function, i.e.\ a locally constant function with compact support. Choose a Haar measure on $\AA^\times$ and define the \emph{Tate integral}
 $$ L(\psi,\chi,s)=\int\limits_{\AA^\times}\psi(a)\chi(a)\norm a^s\,d\!a \;, $$
 which converges for $\Re s>1-\Re\chi$. For every Schwartz-Bruhat function $\psi$ and for every $\chi\in\Xi$, $L(\psi,\chi,s)$ is a holomorphic multiple of $L(\chi,s)$ as function of $s\in\CC$. For every $\chi\in\Xi$, there is a Schwartz-Bruhat function $\psi$ such that $L(\psi,\chi,s)=L(\chi,s)$. In particular if $\chi$ is unramified, then $L(\psi_0,\chi,s)=L(\chi,s)$ for
 $$ \psi_0 \ = \ h\,(q-1)^{-1}\,(\vol\cO_\AA)^{-1}\,\ch_{\cO_\AA} \;. $$
\end{pg}



\begin{pg}
\label{class_field_theory}
 We collect some statements from class field theory. Let $E/F$ be a finite Galois extension and $\N_{E/F}:\AA_E\to\AA_F$ the norm map. The reciprocity homomorphism $r_{E/F}:\Gal(E/F)\to\lquot{F^\times\N_{E/F}(\AA_E^\times)}{\AA_F^\times}$ induces an isomorphism
 $$ r^\ast_{E/F}: \ \Hom(\,\lquot{F^\times\N_{E/F}(\AA_E^\times)}{\AA_F^\times},\ \SS^1\,)\ \longrightarrow\ \Hom(\,\Gal(E/F),\ \SS^1\,) \;. $$
 If $\omega$ is a character of $\Gal(E/F)$, then denote by $\tilde\omega$ the corresponding character of $\AA_F^\times$ that is trivial on $F^\times$ and $\N_{E/F}(\AA_E^\times)$. In particular, since $E/F$ is unramified if and only if $\cO_\AA^\times\subset\N_{E/F}(\AA_E^\times)$, we see that $\tilde\omega$ is unramified if $E/F$ unramified is so.

 \label{product_of_L_series}
 Let $E/F$ be a finite abelian Galois extension and $\chi\in\Xi$. Then
 $$ L_E(\chi\circ\N_{E/F},s) \ = \ \prod_{\omega\in\Hom(\Gal(E/F),\SS^1)} L_F(\chi\tilde\omega,s) $$
 as meromorphic functions of $s$.

 \label{product_of_all_L_series}
 Let $\chi\in\Xi$ be of finite order $n$. Then there is an abelian Galois extension $E/F$ of order $n$ such that $\chi(\N_{E/F}(\AA_E^\times))=1$, and
 $$ \prod_{\omega\in\Hom(\Gal(E/F),\SS^1)} L_F(\chi\circ\tilde\omega,s) = \zeta_E(s) $$
 as meromorphic functions of $s$. If $\chi$ is an unramified character, then $E/F$ is an unramified field extension.

 \label{L_series_has_no_zero_in_0,1}
 Since zeta functions have simple poles in $0$ and $1$, the $L$-series occuring in this product cannot have zeros at $0$ and $1$ if $\chi$ is ramified. This means that if $\chi\in\Xi$ is of finite order and not of the form $\norm \ ^s$ for some $s\in\CC$, then $L(\chi,0)\neq0$ and $L(\chi,1)\neq0$.
\end{pg}



\begin{pg}
 \label{principal_series_representation}
 We introduce principal series representations and Eisenstein series and review some well-known statements. For reference, cf.\ \cite{Bump}, \cite{Gelbart1} and \cite{Li}.

 Let $B$ be the standard Borel subgroup of upper triangular matrices and $\chi\in\Xi$. The \emph{principal series representation $\cP(\chi)$} is the space of all smooth and $K$-finite functions $f\in C^0(G_\AA)$ that satisfy for all $\tinymat a b {} d \in B_\AA$ and all $g\in G_\AA$ the equation
 $$ f\Bigl(\tinymat a b {} d g \Bigr) \ = \ \norm{\frac ad}^{1/2}\chi\Bigl(\frac ad\Bigr)\,f(g). $$
 The right-regular representation of $G_\AA$ restricts from $C^0(G_\AA)$ to $\cP(\chi)$, or, in other words, $\cP(\chi)$ is a subrepresentation of $C^0(G_\AA)$. 
 \label{principal_series_is_irreducibel} \label{isomorphic_principal_series}
 We have $\cP(\chi')\simeq\cP(\chi)$ as representations if and only if either $\chi'=\chi$ or $\chi'=\chi^{-1}$ and $\chi^2\neq\norm \ ^{\pm1}$. The principal series representation $\cP(\chi)$ is irreducible unless $\chi^2=\norm \ ^{\pm 1}$.

 \label{admissibility_of_principal_series}
 A {\it flat section} is a map $f_{\chi}:\CC\to C^0(G_\AA)$ that assigns to each $s\in\CC$ an element $f_\chi(s)\in\cP(\chi\norm \ ^s)$ such that $f_\chi(s)\vert_K$ is independent of $s$. For every $f\in\cP(\chi)$, there exists an unique flat section $f_\chi$ such that $f=f_\chi(0)$. We say $f$ is \emph{embedded in the flat section $f_\chi$}. Note that $f\in\cP(\chi)$ is uniquely determined by its restriction to $K$. 
\end{pg}

\begin{pg}
 \label{def_unram_Eisenstein}\label{continuation_eisenstein}\label{embedding_principal_series_to_automorphic_forms}
 For the remainder of this section, fix $\chi\in\Xi$, $f\in\cP(\chi)$, and $g\in G_\AA$. Since $\chi$ is trivial on $F^\times$, $f\in\cP(\chi)$ is left $B_F$-invariant, and we define
 $$ E(g,f) \ := \ L(\chi^2,1)\ \cdot\hspace{-10pt} \sum_{\gamma \in \lquot{B_F}{G_F}}\hspace{-8pt} f(\gamma g) \;, $$
 If $f$ is embedded in the flat section $f_\chi$, then put 
 $$ E(g,f,s)=E(g,f_\chi(s)), $$ 
 an expression that converges for every $g\in G_\AA$ and $\Re s > 1/2-\Re\chi$. In the domain of cenvergence, $E(g,f,s)$ is analytic as a function of $s$ and has a meromorphic continuation to all $s\in\CC$. It has simple poles in those $s$ for which $\chi^2\norm \ ^{2s}=\norm \ ^{\pm 1}$. The meromorphic continuation $E(\blanc,f)=E(\blanc,f,0)$ in $s=0$ is called the \emph{Eisenstein series associated to $f$}. As a function in the first argument, $E(\blanc,f)$ is an automorphic form, and we have a morphism
 $$ \begin{array}{ccc} \cP(\chi) & \longrightarrow & \cA \\ f    & \longmapsto     & E(\blanc,f) \end{array} $$
 of $\cH$-modules. 

 If $\chi\in\Xi_0$ and $\chi^2\neq\norm \ ^{\pm1}$, then $\cP(\chi)^K$ is $1$-dimensional and contains thus a unique {\it spherical vector}, i.e.\ an $f^0$ such that $f^0(k)=1$ for all $k\in K$. We put $E(g,\chi,s)=E(g,f^0,s)$ and define the \emph{Eisenstein series associated to $\chi$} as $E(g,\chi)=E(g,\chi,0)$.
\end{pg}

\begin{pg}
 \label{def_f_varphi_chi_s}
 Let $\varphi:\AA^2\to\CC$ be a Schwartz-Bruhat function,  i.e.\ a locally constant function with compact support. Choose a Haar measure on $Z_\AA$ and define 
 $$ f_{\varphi,\chi}(s): g \longmapsto \int\limits_{Z_\AA} \varphi((0,1)zg)\chi(\det zg)\norm{\det zg}^{s+1/2}\;dz. $$
 This is a Tate integral and converges for $\Re s > 1/2-\Re\chi$ (cf.\ paragraph \ref{Euler_prod_Tate_integral}). The function $f_{\varphi,\chi}(s)$ is smooth and $K$-finite, and because
 $$ f_{\varphi,\chi}(s)(\tinymat a b {} d g) \ = \ \chi(ad^{-1})\,\norm{ad^{-1}}^{s+1/2}\,f_{\varphi,\chi}(s)(g), $$
 we have $f_{\varphi,\chi}(s)\in\cP(\chi\norm \ ^s)$. 
\end{pg}

\begin{prop}
 \label{def_varphi_0}
 Let $\Re\chi>1$. 
 \begin{enumerate}
  \item\label{prop1} For all $f\in\cP(\chi)$, there exists a Schwartz-Bruhat function $\varphi:\AA^2\to\CC$ such that $f=f_{\varphi,\chi}(0)$.
  \item\label{prop2} If $\chi\in\Xi_0$ and $f^0\in\cP(\chi)$ is the spherical vector, then $f_{\varphi_0,\chi}(0)=L(\chi^2,2s+1)f^0$ for the Schwartz-Bruhat function $\varphi_0 = h\,(q-1)^{-1}\,(\vol\cO_\AA^2)^{-1}\,\ch_{\cO_\AA^2}$.
 \end{enumerate} 
\end{prop}

\begin{proof}
 In \cite[VII.6--VII.7]{Weil}, Weil constructs for every $\chi\in\Xi$ a Bruhat-Schwartz function $\varphi$ such that $f_{\varphi,\chi}(0)$ is nontrivial. For a proof of \eqref{prop2} observe that for $g=e$,
 $$ f_{\varphi_0,\chi}(0)(e) \ = \ \int\limits_{Z_\AA} \varphi_0((0,1)z)\chi(\det z)\norm{\det z}^{s+1/2}\;dz, $$
 which is the Tate integral for $L(\chi^2,2s+1)$ (cf.\ paragraph \ref{Euler_prod_Tate_integral}).

 For a proof of \eqref{prop1} observe that $\varphi_g=\varphi(\blanc g)$ is still a Schwartz-Bruhat function for every $g\in G_\AA$, and $g.f_{\varphi,\chi}(0)=f_{\varphi_g,\chi}(0)$ is still a function in $\cP(\chi)$. As explained in paragraph \ref{principal_series_is_irreducibel}, $\Re\chi>1$ implies that $\cP(\chi)$ is irreducible, and thus $G_\AA.f_{\varphi,\chi}(0)=\cP(\chi)$.
\end{proof}

\begin{pg}
 \label{E_varphi}
 Define 
 $$ E(g,\varphi,\chi,s) \ = \ \sum_{\gamma\in\lquot{B_F}{G_F}}\ f_{\varphi,\chi}(s)(\gamma g) $$
 for $\Re s>1/2-\Re\chi$. This definition extends to a meromorphic function of $s\in\CC$. Put $E(g,\varphi,\chi)=E(g,\varphi,\chi,0)$. The last proposition implies that the class of Eisenstein series of the form $E(\blanc,\varphi,\chi)$ is the same as the class of Eisenstein series of the form $E(\blanc,f)$. For $\chi\in\Xi_0$, we obtain the equality $E(\blanc,\varphi_0,\chi,s)=E(\blanc,\chi,s)$. 
\end{pg}


\section{The non-split torus case}
\label{section_non-split}

\begin{pg}
 Let $E$ be a separable quadratic field extension of $F$. Consider a non-split torus $T\subset G$, whose $F$-rational points are the image of $E^\times$ under an injective homomorphism of algebras $\Theta_E: E \to \Mat_2(F) $ given by the choice of a basis of $E$ over $F$. This homomorphism extends to $\Theta_E: \AA_E^\times \to G_{\AA_F} $. Let $\N_{E/F}:\AA_E^\times \to \AA_F^\times$ be the norm. We have that $\det(\Theta_E(t))=\N_{E/F}(t)$ (\cite[Prop.\ VI.5.6]{Lang2}).
 
 Let $h_E$ denote the class number of $E$ and let $q_E$ be the cardinality of the constant field of $E$. Consider the $\AA_F$-linear projection
 $$ \begin{array}{cccc}
     \pr: & \Mat_2 \AA_F & \longrightarrow & \AA_F^2 \;. \\ 
	       &    g       & \longmapsto     & (0,1)g
    \end{array} $$
 The kernel of $\pr$ is contained in the upper triangular matrices and does not contain any nontrivial central matrix. The intersection of the upper triangular matrices with $T_{\AA}$ is $Z_\AA$. Thus $\Theta_E(\AA_E)\cap\ker\pr=\{0\}$ and the $\AA_F$-linear map $\widetilde\Theta_E=\pr\circ\Theta_E: \AA_E\to\AA_F^2$ is injective. This implies that $\widetilde\Theta_E$ is an isomorphism of $\AA_F$-modules.
 
 In the natural topology as free $\AA_F$-modules, $\widetilde\Theta_E$ is thus a isomorphism of locally compact groups. Define $\varphi_T:\AA_F^2\to\CC$ as $h_E(q_E-1)^{-1}(\vol\cO_{\AA_E})^{-1}$ times the characteristic function of $\widetilde\Theta_E({\cO_{\AA_E}})$. Since $\widetilde\Theta_E$ is a homeomorphism, $\varphi_T$ and also $\varphi_{T,g}=\varphi_T(\blanc g)$ are Schwartz-Bruhat functions for all $g\in G_\AA$.
\end{pg}
 
\begin{lemma}
 \label{E^i_varphi=integral_sum}
 For $\Re s>1/2-\Re\chi$,
  $$ E(g,\varphi,\chi,s) \ = \ \int\limits_{\lquot{Z_F}{Z_\AA}} \sum_{u \in F^2-\{0\}} \varphi(uzg) \chi(\det zg) \norm{\det zg}^{s+1/2} \,d\!z \;. $$
\end{lemma}

\begin{proof}
 Let $G_F$ act on $\PP^1(F)$ by multiplication from the right. Then $B_F$ is the stabiliser of $[0:1]$, and thus we have a bijection
 $$ \begin{array}{ccl}
     \lquot{B_F}{G_F} & \overset{1:1}{\longrightarrow} & \PP^1(F) \ = \ \lquot{Z_F}{(F^2-\{0\})} \;. \\
            g         & \longmapsto                    & [0:1] g
	 \end{array} $$
 Since $\sum_{\gamma \in \lquot{B_F}{G_F}} f(\gamma g)$ is absolutely convergent for every $f\in\cP(\chi\norm \ ^s)$ and $g\in G_\AA$, (\cite[Thm.\ 2.3]{Li}), the lemma follows by Fubini's theorem.
\end{proof}

The following is a refinement of Zagier's translation of a formula of Hecke into adelic language (\cite[pp.\ 298-299]{Zagier1}).

\begin{thm}
 \label{thm_Zagier_anis}
 Let $T$ be a non-split torus corresponding to a separable field extension $E/F$. For every $\varphi:\AA^2\to\CC$ that is a Schwartz-Bruhat function, $g\in G_\AA$ and $\chi\in\Xi$, there exists a holomorphic function $e_T(g,\varphi,\chi,s)$ of $s\in\CC$ with the following properties.
 \begin{enumerate}
  \item\label{anis1} For all $s\in\CC$ such that $\chi^2\norm \ ^{2s}\neq\norm \ ^{\pm1}$,
       $$ E_T(g,\varphi,\chi,s) \ = \ e_T(g,\varphi,\chi,s)\ L_E(\chi\circ\N_{E/F},s+1/2) \;. $$
  \item\label{anis2} For every $g\in G_\AA$ and $\chi\in\Xi$, there is a Schwartz-Bruhat function $\varphi:\AA^2\to\CC$ such that
                     $$ e_T(g,\varphi,\chi,s) \ = \ \chi(\det g)\norm{\det g}^{s+1/2} $$
		     for all $s\in\CC$.
 \end{enumerate}
\end{thm}

\begin{proof}
 For every Schwartz-Bruhat function $\varphi:\AA^2\to\CC$, $g\in G_\AA$ and $\chi\in\Xi$, both $E_T(g,\varphi,\chi,s)$ and $L_E(\chi\circ\N_{E/F},s+1/2)$ are meromorphic functions of $s\in\CC$. Define $e_T(g,\varphi,\chi,s)$ as their quotient. This is a meromorphic function in $s$ that satisfies \eqref{anis1}. We postpone the proof that $e_T(g,\varphi,\chi,s)$ is holomorphic in $s$ to the very end and continue with showing that $e_T(g,\varphi,\chi,s)$ satisfies part \eqref{anis2}.

 Note that our choices of Haar measures fit the applications of Fubini's theorem in the following calculations. Let $\Re s > 1/2-\Re\chi$, then Lemma \ref{E^i_varphi=integral_sum} applies, and we obtain
 $$ E_T(g,\varphi,\chi,s) \ = \hspace{-6pt}\int\limits_{\lquot{T_FZ_\AA}{T_\AA}}\quad\int\limits_{\lquot{Z_F}{Z_\AA}} \sum_{u \in F^2-\{0\}} \varphi(uztg)\,\chi(\det(ztg))\,\norm{\det(ztg)}^{s+1/2}\,dz\,dt \;. $$
 Since $\lquot{T_F}{T_\AA} \simeq (\lquot{T_FZ_\AA}{T_\AA}) \times \bigl(\lquot{Z_F}{Z_\AA}\bigr)$, we can apply Fubini's theorem to derive
 $$ E_T(g,\varphi,\chi,s) \ = \ \int\limits_{\lquot{T_F}{T_\AA}} \sum_{u \in F^2-\{0\}} \varphi(utg)\,\chi(\det(tg))\,\norm{\det(tg)}^{s+1/2} \,dt \;. $$
 The map $\Theta_E$ identifies $\AA_E^\times$ with $T_{\AA_F}$. The $\AA_F$-linear isomorphism $\widetilde\Theta_E$ identifies $\AA_E$ with $\AA_F^2$ and restricts to a bijection between $E^\times$ and $F^2-\{0\}$. Thus we can rewrite the integral as
 $$ \chi(\det g)\,\norm{\det g}^{s+1/2} \int\limits_{\lquot{E^\times}{\AA_E^\times}} \sum_{u\in E^\times} \varphi(\widetilde\Theta_E(ut)g)\,\chi(\N_{E/F}(t))\,\norm{\N_{E/F}(t)}^{s+1/2} \,dt \;. $$
 If we define $\tilde\varphi_g=\varphi\bigl(\widetilde\Theta_E(\blanc)g\bigr):\AA_E\to\CC$ and apply Fubini's theorem again, we get
 $$ E_T(g,\varphi,\chi,s) \ = \ \chi(\det g)\,\norm{\det g}^{s+1/2} \int\limits_{\AA_E^\times} \tilde\varphi_g(t)\,\chi\circ\N_{E/F}(t)\,\norm{t}_{\AA_E}^{s+1/2} \,dt \;. $$
 Note that $\tilde\varphi_g:\AA_E\to\CC$ is a Bruhat-Schwartz function as $\varphi$ is one. Thus the integral is the Tate integral $L_E(\tilde\varphi_g,\chi\circ\N_{E/F},s+1/2)$. There is a Schwartz-Bruhat function $\psi:\AA_E\to\CC$ such that 
 $$ L_E(\psi,\chi\circ\N_{E/F},s+1/2) \ = \ L_E(\chi\circ\N_{E/F},s+1/2) $$
 (paragraph \ref{Euler_prod_Tate_integral}). If we define $\varphi:\AA_F^2\to\CC$ to be the Schwartz-Bruhat function such that $\psi=\tilde\varphi_g$, then $e_T(g,\varphi,\chi,s)=\chi(\det g)\,\norm{\det g}^{s+1/2}$. If $\chi\in\Xi_0$, then $\chi\circ\N_{E/F}$ is an unramified character of $\AA_E$ and
 $$ \psi=\varphi_{T,g^{-1}}\bigl(\widetilde\Theta_E(\blanc)g\bigr)=\varphi_T\circ\widetilde\Theta_E $$
 yields the desired $\psi$ as it adopts the role of $\psi_0$ in paragraph \ref{Euler_prod_Tate_integral}, and part \eqref{anis2} is proven. 
 
 Since $L_E(\psi,\chi\circ\N_{E/F},s+1/2)$ equals a holomorphic multiple of $L_E(\chi\circ\N_{E/F},s+1/2)$ in $s\in\CC$ for any Schwartz-Bruhat function $\psi=\tilde\varphi_g$ (paragraph \ref{Euler_prod_Tate_integral}), we finally see that the function $e_T(g,\varphi,\chi,s)$ is holomorphic in $s$.
\end{proof}

 By the definition of $E$-toroidality, we obtain as an immediate consequence:

\begin{cor}
 \label{cor_anis}
 Let $\chi\in\Xi$ such that $\chi^2\neq\norm \ ^{\pm1}$ and let $\varphi:\AA^2\to\CC$ be a Schwartz-Bruhat function. Let $E/F$ be a separable quadratic field extension. Then $E(\blanc,\varphi,\chi)$ is $E$-toroidal if and only if $L_E(\chi\circ\N_{E/F},1/2)=0$.\qed
\end{cor}


\section{The split torus case}
\label{section_split}

\begin{pg}
 \label{varphi_T_for_diagonal_torus}
 In this section, we establish the analogue of Theorem \ref{thm_Zagier_anis} for split tori, which is also the adelic translation of a long-known formula (\cite[eq.\ (30)]{Zagier1}). To begin with, let $T=\{\tinymat {\ast} {} {} {\ast} \}\subset G$  be the diagonal torus. We write $\AA$ for the adeles of $F$. Define the Schwartz-Bruhat function $\varphi_T:\AA^2\to\CC$ as $h(q-1)^{-1}(\vol\cO_{\AA})^{-1}$ times the characteristic function of $\cO_\AA^2$, which is the same as $\varphi_0$ as defined in Proposition \ref{def_varphi_0}. Put $\varphi_{T,g}=\varphi_T(\blanc g)$, which is a Schwartz-Bruhat function since multiplying with $g$ from the right is an automorphism of the locally compact group $\AA_F^2$. Recall from paragraph \ref{def_f_varphi_chi_s} that we defined
 $$ f_{\varphi,\chi}(s)(g) = \int\limits_{Z_\AA} \varphi((0,1)zg)\chi(\det(zg))\norm{\det(zg)}^{s+1/2}\;dz $$
 for $\Re s>1/2-\Re\chi$. Put $e=\tinymat 1 {} {} 1 $ and $w_0=\tinymat {} 1 1 {} $.
\end{pg}

\begin{lemma}
 \label{lemma_zagier_diag}
 Let $T$ be the diagonal torus. For every $\varphi:\AA^2\to\CC$ that is a Schwartz-Bruhat function, $g\in G_\AA$ and $\chi\in\Xi$, there exists a holomorphic function $\tilde e_T(g,\varphi,\chi,s)$ of $s\in\CC$ with the following properties.
 \begin{enumerate}
  \item\label{diag1} For all $s\in\CC$ such that $\chi^2\norm \ ^{2s}\neq\norm \ ^{\pm1}$,
        \begin{multline*} 
        \int\limits_{\lquot{T_FZ_\AA}{T_\AA}}\bigl(E(tg,\varphi,\chi,s)-f_{\varphi,\chi}(s)(tg)-f_{\varphi,\chi}(s)(w_0tg)\bigr)\,dt\\
        = \ \tilde e_T(g,\varphi,\chi,s)\ \bigl(L(\chi,s+1/2)\bigr)^2 \;.
        \end{multline*} 
        In particular, the left hand side is well-defined and converges.
  \item\label{diag2} For every $g\in G_\AA$ and $\chi\in\Xi$, there is a Schwartz-Bruhat function $\varphi:\AA^2\to\CC$ such that
  							$$ \tilde e_T(g,\varphi,\chi,s) \ = \ \chi(\det g)\norm{\det g}^{s+1/2} $$
							for all $s\in\CC$. If $\chi\in\Xi_0$, then $\varphi=\varphi_{T,g^{-1}}$ satisfies the equation.
 \end{enumerate}
\end{lemma}

\begin{proof}
 Let $\Re s>1/2-\Re\chi$, and denote the left hand side of the equation in \eqref{diag1} by $I$. Note that our choices of Haar measures match with the following applications of Fubini's theorem. We choose $\left\{ e,w_0,\tinymat 1 {} c 1 \right\}_{c \in F^\times}$ as a system of representatives of $B_F \backslash G_F$. By definition of $E(tg,\varphi,\chi,s)$,
 $$ E(tg,\varphi,\chi,s)-f_{\varphi,\chi}(s)(tg)-f_{\varphi,\chi}(s)(w_0tg)
    \ = \ \sum_{c\in F^\times} \, f_{\varphi,\chi}(s)\Bigl(\tinymat 1 {} c 1 tg\Bigr) \;. $$
 Hence
 $$ I \ = \ \int\limits_{\lquot{T_FZ_\AA}{T_\AA}} \sum_{c\in F^\times} \, f_{\varphi,\chi}(s)\Bigl(\tinymat 1 {} c 1 tg\Bigr)\,dt \;. $$
 Note that this is a well-defined expression since
 $$ f_{\varphi,\chi}(s)\Bigl(\tinymat 1 {} c 1 \ \tinymat zt_1 {} {} zt_2 \Bigr) \ = \ f_{\varphi,\chi}(s)\Bigl(\tinymat zt_1 {} {} zt_2 \ \tinymat 1 {} ct_1t_2^{-1} 1 \Bigr) \ = \ f_{\varphi,\chi}(s)\Bigl(\tinymat 1 {} ct_1t_2^{-1} 1 \Bigr) $$
 for $\tinymat zt_1 {} {} zt_2 \in T_FZ_\AA$, so changing the representative of $t\in\lquot{T_FZ_\AA}{T_\AA}$ only permutes the set $\left\{\tinymat 1 {} c 1 \right\}_{c \in F^\times}$. Inserting the definition of $f_{\varphi,\chi}(s)$ yields
 $$ I \ = \ \int\limits_{\lquot{T_FZ_\AA}{T_\AA}} \sum_{c\in F^\times} \int\limits_{Z_\AA} \varphi\bigl((c,1)ztg\bigr)\ \chi(\det(ztg))\,\norm{\det(ztg)}^{s+1/2}\,dz\,dt \;. $$
 By writing $\varphi_g$ for the Schwartz-Bruhat function $\varphi(\blanc g)$, applying Fubini's theorem to 
 $$ \bigl(\lquot{T_F Z_{\AA}}{T_{\AA}}\bigr) \times Z_{\AA} \ \simeq \ \bigl(\lquot{T_F}{T_{\AA}}\bigr) \times Z_F $$ 
 and observing that we have $\det z\in F^\times\subset\ker(\chi\norm \ ^{s+1/2})$ for a matrix $z \in Z_F$, we find
 $$ I \ = \ \int\limits_{\lquot{T_F}{T_{\bf A}}} \ \sum_{c\in F^\times} \ \int\limits_{F^\times} \ 
        \varphi_g((zc,z)t)\ \chi(\det(tg))\ \norm{\det(tg)}^{s+1/2} \,dz\,dt \;. $$ 
 We now replace $c$ by $cz^{-1}$, replace the sum by the integral over the discrete space $F^\times$ and use 
 $$ \begin{array}{ccc}
     \lquot{T_F}{T_{\bf A}} &\simeq& (\lquot{F^\times}{{\AA}^\times}) \times (\lquot{F^\times}{{\AA}^\times}) \;. \\
     t &\mapsto& (t_1,t_2) 
	 \end{array} $$
 Then $I$ equals
 \begin{multline*} 
     \hspace{-8pt} \chi(\det g)\ \norm{\det g}^{s+1/2} \hspace{-8pt} \int\limits_{\lquot{F^\times}{{\bf A}^\times}} \ \int\limits_{\lquot{F^\times}{{\bf A}^\times}} \ \int\limits_{F^\times} \ \int\limits_{F^\times} \varphi_g(ct_1,at_2)\ \chi(t_1t_2)\ \norm{t_1t_2}^{s+1/2} \,da\,dc\,dt_1\,dt_2 \\ 
     \displaystyle = \ \chi(\det g)\ \norm{\det g}^{s+1/2} \int\limits_{{\bf A}^\times}\ \biggl(\ \ \int\limits_{{\bf A}^\times} \ \varphi_g(t_1,t_2 )\ \chi(t_1)\ \norm{t_1}^{s+1/2} \,dt_1\ \biggr)\ \chi(t_2)\ \norm{t_2}^{s+1/2} \,dt_2 \;.
 \end{multline*}

 Let $U\subset\AA^2$ be the compact domain of $\varphi_g$. Then $\bigl\{t_1\in\AA\bigl|\bigl(\{t_1\}\times\AA\bigr)\cap U\neq\emptyset\bigr\}$ is compact. For every $t_2$, the function $t_1\to\varphi_g(t_1,t_2)$ is locally constant on $\AA\times\{t_2\}\subset\AA\times\AA$ endowed with the subspace topology. Consequently, $\varphi_g(\blanc,t_2)$ is a Schwartz-Bruhat function for every $t_2$ and the expression in brackets that we see in the last equation is a Tate integral, which equals a multiple of $L(\chi,s+1/2)$ (cf.\ paragraph \ref{Euler_prod_Tate_integral}). Denote the factor by $\tilde \varphi_g(t_2)$. For the same reasons as before, but with the roles of $t_1$ and $t_2$ reversed, we see that $\varphi_g(t_1,\blanc)$ is a Schwartz-Bruhat function for every $t_1$. Hence the value of the Tate integral is locally constant in $t_2$ and vanishes at all $t_2$ outside a compact set. Since $L(\chi,s+1/2)$ does not depend on $t_2$, the factor $\tilde \varphi_g$ is locally constant and compact support. Hence $\tilde \varphi_g:\AA\to\CC$ is a Schwartz-Bruhat function. Substituting the Tate integral in the last equation by $\tilde \varphi_g(t_2)L(\chi,s+1/2)$ yields
 $$ I \ = \ \chi(\det g)\ \norm{\det g}^{s+1/2} L(\chi,s+1/2) \int\limits_{{\AA}^\times} \tilde \varphi_g(t_2) \chi(\det g) \norm{t_2}^{s+1/2} \,dt_2 \;, $$
 where we see again a Tate integral, which equals a multiple of $L(\chi,s+1/2)$. 
 
 We end up with the right hand side of the equation in \eqref{diag1} if $\tilde e_T(g,\varphi,\chi,s)$ is suitably defined. In particular, the left hand side is a well-defined and converging expression, which is meromorphic in $s\in\CC$, and $\tilde e_T(g,\varphi,\chi,s)$ is meromorphic as the quotient of meromorphic functions. Hence \eqref{diag1} holds.
 
 There is a Schwartz-Bruhat function $\psi:\AA\to\CC$ such that we have $L(\psi,\chi,s+1/2) = L(\chi,s+1/2)$ (cf.\ paragraph \ref{Euler_prod_Tate_integral}). If we define $\varphi:\AA^2\to\CC$ to be the Schwartz-Bruhat function such that $\varphi_g(t_1,t_2)=\psi(t_1)\cdot\psi(t_2)$. Then $\tilde e_T(g,\varphi,\chi,s)=\chi(\det g)\norm{\det g}^{s+1/2}$. If $\chi\in\Xi_0$, then $\varphi_{T,g^{-1}}$ satisfies the equality (cf.\ paragraph \ref{Euler_prod_Tate_integral}). Hence \eqref{diag2} holds by meromorphic continuation. 
 
 The Tate integral $L(\psi,\chi,s+1/2)$ equals a holomorphic multiple of $L(\chi,s+1/2)$ in $s\in\CC$ for any Schwartz-Bruhat function $\psi$ (cf.\ paragraph \ref{Euler_prod_Tate_integral}), thus $\tilde e_T(g,\varphi,\chi,s)$ is holomorphic in $s\in\CC$ for an arbitrary Schwartz-Bruhat function $\varphi$.
\end{proof}

\begin{pg}
 \label{Fourier_dec_of_Eisenstein_series} 
 We state the functional equation for Eisenstein series. For reference, see \cite{Li}. Let $B$ be the standard Borel subgroup and $N$ its unipotent radical. The constant term of $E(g,\varphi,\chi,s)$ is given by
 $$ E_N(g,\varphi,\chi,s) \ = \ f_{\varphi,\chi}(s)(g) \ + \ M_\chi(s)\,f_{\varphi,\chi}(s)(g) $$ 
 with
 $$ M_\chi(s)f_{\varphi,\chi}(g) \ = \ \int\limits_{N_\AA}f_{\varphi,\chi}(\tinymat {} 1 1 b g)\;db \;. $$
 Note that the operator $M_\chi(s)$ is a morphism of $G_\AA$-modules $\cP(\chi\norm \ ^s) \rightarrow \cP(\chi^{-1}\norm \ ^{-s})$, which is defined for all $s$ unless $\chi^2\norm\ ^{2s}=1$.

 \label{factor_of_functional_equation}
 Let $f_{\varphi,\chi}$ be embedded in the flat section $f_{\varphi,\chi}(s)$ and let $M_\chi(0)f_{\varphi,\chi} \in \cP(\chi^{-1})$ be embedded in the flat section $\hat f_{\varphi,\chi^{-1}}(s)$. Then there is a holomorphic function $c(\chi,s)$ in $s\in\CC$ such that
 $$ M_\chi(s)\ f_{\varphi,\chi} \ = \ c(\chi,s) \ \hat f_{\varphi,\chi^{-1}}(-s) \; $$
 for all $\chi\in\Xi$ and $s\in\CC$ unless $\chi^2\norm \ ^{2s}=1$.  If $\chi\in\Xi_0$, then $c(\chi,s)=\chi^2(\fc) \norm \fc^{2s}$.
 \label{functional_equation_eisenstein}
 For $\chi^2\norm \ ^{2s}\neq\norm \ ^{\pm1}$, this yields the {\rm functional equation}
 $$ E\bigr(\blanc,f_{\varphi,\chi}(s)\bigl) \ = \ c(\chi,s)\, E\bigr(\blanc,\hat f_{\varphi,\chi^{-1}}(-s)\bigl). $$

 By paragraph \ref{E_varphi}, there is a Schwartz-Bruhat function $\hat\varphi$ such that 
 $$ \hat f_{\varphi,\chi^{-1}}(s) \ = \ f_{\hat\varphi,\chi^{-1}}(s) \hspace{1cm}\text{and}\hspace{1cm} E(\blanc,\hat\varphi,\chi^{-1},s) \ = \ E(\blanc,\hat f_{\varphi,\chi^{-1}}(s)) \;. $$
 Further recall from paragraph \ref{cont_L_series} the functional equation 
 $$ L(\chi,1/2+s) = \epsilon(\chi,s) L(\chi^{-1},1/2-s) $$
 of $L$-series where $\epsilon(\chi,s)=\chi(\fc)\norm{\fc}^s$ if $\chi$ is unramified.
\end{pg}
 
\begin{pg}
 Let $T\subset G$ be a split torus. Then $T_F$ is given as the image of $\Theta_E: E^\times \to G_F $, where $E=F\oplus F$. We recall the definition of $f_T$ for split tori (see paragraph \ref{def_toroidal_integral}), which is
 $$ f_T(g) \ = \ \int \limits_{\lquot{T_FZ_\AA}{T_\AA}}\Bigl(f-\frac 12(f_N+f_{N^T})\Bigr)\,(tg)\,dt $$
 for $f\in\cA$, where
 $$ f_{N^T}(g) \ = \ \int\limits_{\lquot{N^T_F}{N^T_\AA}} f(ng)\,dn \ = \ \int\limits_{\lquot{N_F}{N_\AA}} f(w_0nw_0 g)\,dn \ = \ f_N(w_0g) \;. $$ 

 As remarked in paragraph \ref{bijection_quadr_ext_conj_cl_of_tori}, there is a $\gamma\in G_F$ such that $T=\gamma^{-1}T_0\gamma$, where $T_0$ is the diagonal torus. We defined $\varphi_{T_0}$ for the diagonal torus $T_0$ in paragraph \ref{varphi_T_for_diagonal_torus}. Define $\varphi_{T}=\varphi_{T_0,\gamma}$. Note that this definition does not depend on $\gamma$ because the only matrices that leave $T_0$ invariant by conjugation are $e$ and $w_0$. But $\varphi_{T_0}(\blanc w_0\gamma)=\varphi_{T_0}(\blanc \gamma)$ by the definition of $\varphi_{T_0}$.
\end{pg}

\begin{thm}
 \label{thm_Zagier_split} 
 Let $T$ be a split torus. For every Schwartz-Bruhat function $\varphi:\AA^2\to\CC$, $g\in G_\AA$ and $\chi\in\Xi$, there exists a holomorphic function $e_T(g,\varphi,\chi,s)$ of $s\in\CC$ with the following properties.
 \begin{enumerate}
  \item\label{split1} For all $s\in\CC$ such that $\chi^2\norm \ ^{2s}\neq\norm \ ^{\pm1}$,
                      $$ \label{split_toroidal_integral_of_Eisenstein_series} E_T(g,\varphi,\chi,s) \ = \ e_T(g,\varphi,\chi,s)\ \bigl(L(\chi,s+1/2)\bigr)^2 \;. $$
  \item\label{split2} If $\chi\in\Xi_0$, then $e_T(e,\varphi_T,\chi,s)=1$ for all $s\in\CC$.
 \end{enumerate}
\end{thm}

\begin{proof}
 First, let $T$ be the diagonal torus. Let $\chi\in\Xi_0$ and $s\in\CC$ such that $\chi^2\norm \ ^{2s}\neq\norm \ ^{\pm1}$. We calculate:
 $$ \begin{array}{l}
  \displaystyle 2\,E_T(g,\varphi,\chi,s) \vspace{5pt}\\ 
  \displaystyle = \ \int \limits_{\lquot{T_FZ_\AA}{T_\AA}}\bigl(2\,E(tg,\varphi,\chi,s) \ - \ E_N(tg,\varphi,\chi,s) \ - \ E_{N^T}(tg,\varphi,\chi,s)\bigr)\,dt \vspace{5pt}\\ 
  \displaystyle = \ \int \limits_{\lquot{T_FZ_\AA}{T_\AA}}\Bigl(2\,E(tg,\varphi,\chi,s) \ - \ f_{\varphi,\chi}(s)(tg) \ - \ M_\chi(s)\,f_{\varphi,\chi}(s)(tg) \vspace{-8pt}\\ 
  \displaystyle     \hspace{5,15cm} - \ f_{\varphi,\chi}(s)(w_0tg) \ - \ M_\chi(s)\,f_{\varphi,\chi}(s)(w_0tg)\Bigr) \,dt \vspace{3pt}\\ 
  \displaystyle = \ \int \limits_{\lquot{T_FZ_\AA}{T_\AA}} \Bigl(\bigl(E(tg,\varphi,\chi,s) \ - \ f_{\varphi,\chi}(s)(tg)\ - \ f_{\varphi,\chi}(s)(w_0tg)\bigr) \\  
  \displaystyle     \hspace{1cm} + \ c(\chi,s)\ \bigl(E(tg,\hat\varphi,\chi^{-1},-s) \ - \ f_{\hat\varphi,\chi^{-1}}(-s)(tg) \ - \ f_{\hat\varphi,\chi^{-1}}(-s)(w_0tg)\bigr)\Bigr)\,dt \;, \vspace{-5pt}
 \end{array}$$
 where we applied the formulas of the previous paragraph and the functional equation (paragraph \ref{functional_equation_eisenstein}). By Lemma \ref{lemma_zagier_diag}, we can split the last integral into two and obtain:
 $$ \tilde e_T(g,\varphi,\chi,s)\ \bigl(L(\chi,s+1/2)\bigr)^2 \ + \ c(\chi,s)\ \tilde e_T(g,\hat\varphi,\chi^{-1},-s)\ \bigl(L(\chi^{-1},-s+1/2)\bigr)^2 \;. $$
 We apply the functional equation to $L(\chi^{-1},-s+1/2)$ and obtain \eqref{split1} for the diagonal torus if we put
 $$ e_T(g,\varphi,\chi,s) \ = \ \frac 12\ \tilde e_T(g,\varphi,\chi,s)\ + \ \frac 12\ \epsilon(\chi,s)^{-2}\ c(\chi,s)\ \tilde e_T(g,\hat\varphi,\chi^{-1},-s) \;. $$
 This defines $e_T(g,\varphi,\chi,s)$ as a holomorphic function of $s\in\CC$ since $\epsilon(\chi,s)$ is non-vanishing as a function at $s$.
	 
 If $T$ is any split torus, define $e_T(g,\varphi,\chi,s)=e_{T_0}(\gamma g,\varphi,\chi,s)$. Since all split tori in $G$ are conjugated, there is a $\gamma\in G_F$ such that $T=\gamma T_0 \gamma^{-1}$, where $T_0$ is the diagonal torus. Recall from paragraph \ref{tor_def_well-defined} that $f_{T}(g) = f_{T_0}(\gamma g)$. This reduces the case of the general split torus to the case of the diagonal torus. Thus \eqref{split1} holds.

 Regarding \eqref{split2}, let $\chi\in\Xi_0$ and $s\in\CC$ be such that $\chi^2\norm \ ^{2s}\neq\norm \ ^{\pm1}$. Since we may replace $\chi$ by $\chi\norm \ ^s$, we assume that $s=0$ without loss of generality. Recall from paragraph \ref{E_varphi} that $E(\blanc,\varphi_0,\chi)=E(\blanc,\chi)$. Put $f_\chi=f_{\varphi_0,\chi}(0)\in\cP(\chi)$ and $f_{\chi^{-1}}=f_{\varphi_0,\chi^{-1}}(0)\in\cP(\chi^{-1})$. By paragraph \ref{Fourier_dec_of_Eisenstein_series}, we have
 $$ E_N(g,\chi) \, = \, f_\chi(g)+M_\chi(0)f_\chi(g) \quad\text{and}\quad E_N(g,\chi^{-1}) \, = \, f_{\chi^{-1}}(g)+M_{\chi^{-1}}(0)f_{\chi^{-1}}(g) \;, $$
 where $N$ is the unipotent radical of the standard Borel subgroup.
 
 Observe that for $T=\gamma^{-1} T_0\gamma$, we have $e_T(e,\varphi_T,\chi,s)=e_{T_0}(\gamma,\varphi_{T_0,\gamma^{-1}},\chi,s)$. As in the proof of \eqref{split1}, we may restrict to the diagonal torus $T=T_0$ without loss of generality. We follow the lines of the calculation in the proof of \eqref{split1}, where we make use of the functional equation for $E(\blanc,\chi)$ (paragraph \ref{functional_equation_eisenstein}), the functional equation for $L(\chi,1/2)$ (paragraph \ref{cont_L_series}) and Lemma \ref{lemma_zagier_diag} \eqref{diag2}:
 \begin{align*}
  2\,E_T(e,\chi) \ &= \ \int \limits_{\lquot{T_FZ_\AA}{T_\AA}}\bigl(2\,E(t,\chi) \ - \ E_N(t,\chi) \ - \ E_{N^T}(t,\chi)\bigr)\,dt \\
  &= \ \int \limits_{\lquot{T_FZ_\AA}{T_\AA}} \Bigl(\bigl(E(t,\chi) \ - \ f_{\chi}(t)\ - \ f_{\chi}(w_0t)\bigr) \\ 
  &\hspace{2cm} + \ \chi^2(\fc)\ \bigl(E(t,\chi^{-1}) \ - \ f_{\chi^{-1}}(t) \ - \ f_{\chi^{-1}}(w_0t)\bigr)\Bigr)\,dt \;, \\
  &= \ \tilde e_T(e,\varphi_T,\chi,0)\, \bigl(L(\chi,1/2)\bigr)^2 \ + \ \tilde e_T(e,\varphi_T,\chi^{-1},0)\,\chi^2(\fc)\, \bigl(L(\chi^{-1},1/2)\bigr)^2 \\ 
  &= \ \bigl(L(\chi,1/2)\bigr)^2 \ + \ \chi^2(\fc)\,\chi^{-2}(\fc)\,\bigl(L(\chi,1/2)\bigr)^2 \\ 
  &= \ 2\,\bigl(L(\chi,1/2)\bigr)^2 \;.
 \end{align*}
 By holomorphic continuation, we find $e_T(e,\varphi_0,\chi,s)=1$ for all $s\in\CC$.
\end{proof}

 For any Schwartz-Bruhat function $\varphi:\AA^2\to\CC$ and any $g\in G_\AA$ we have that the automorphic form $g.E(\blanc,\varphi,\chi)$ is an element of $\cP(\chi)$ (paragraph \ref{E_varphi}). By the definition of $F\oplus F$-toroidality, we obtain as an immediate consequence:

\begin{cor}
 Let $\chi\in\Xi_0$ such that $\chi^2\neq\norm \ ^{\pm1}$ and let $\varphi:\AA^2\to\CC$ be a Schwartz-Bruhat function. Then $E(\blanc,\varphi,\chi)$ is $F\oplus F$-toroidal if and only if $L(\chi,1/2)=0$. \qed
\end{cor}


\section{Toroidal derivatives of Eisenstein series}
\label{section_derivatives}

\begin{pg}
 \label{def_derivatives}
  Fix an $i\geq0$, a Schwartz-Bruhat function $\varphi:\AA^2\to\CC$ and a $\chi\in\Xi$ such that $\chi^2\neq\norm\ ^{\pm1}$. We define $E^{(i)}(g,\varphi,\chi,s)$ as the \emph{$i$-th derivative of $E(g,\varphi,\chi,s)$ with respect to $s$}. The function $E^{(i)}(\cdot,\varphi,\chi,s)$ is an automorphic form.  We denote by $L^{(i)}(\chi,s)$ the \emph{$i$-th derivative of $L(\chi,s)$ with respect to $s$}. 

 \label{def_chi_T}
 Let $T\subset G$ be a maximal torus defined by $\Theta_E: E^\times\to G_F$. If $E$ is a field, then the reciprocity map (cf.\ paragraph \ref{class_field_theory}) assigns to the nontrivial character of $\Gal(E/F)$ a character of $\AA_F^\times$, which we denote by $\chi_T=\chi_E$. This character is of order two and its kernel is precisely $\N_{E/F}(\AA_E^\times)$. Further
 $$ L_E(\chi\circ\N_{E/F},s) \ = \ L_F(\chi,s)\ L_F(\chi\chi_T,s) \;. $$
 If $E=F\oplus F$, then we define $\chi_T=\chi_E$ as the trivial character. For every maximal torus $T$ of $G$, we put
 $$ e_T^{(i)}(g,\varphi,\chi) \ := \ \frac{d^i}{ds^i} \, e_T(g,\varphi,\chi,s)\Bigr|_{s=0} \;. $$
\end{pg}

\begin{thm}
 \label{lemma_zagier_for_derivatives}
 Let $T$ be a maximal torus in $G$ and $n$ a positive integer. For all $g\in G_\AA$ and $\chi\in\Xi_0$ such that $\chi^2\neq\norm \ ^{\pm1}$,
 $$ E_T^{(n)}(g,\varphi,\chi) \ = \ \sum_{\substack{i+j+k=n\\i,j,k\geq0}} \frac{n!}{i!\,j!\,k!}\ e_T^{(i)}(g,\varphi,\chi)\ L^{(j)}(\chi,1/2)\ L^{(k)}(\chi\chi_T,1/2) \;. $$
 In particular, $E^{(n-1)}(\blanc,\varphi,\chi)$ is toroidal if $L(\chi,s+1/2)$ vanishes in $s=0$ to order at least $n$.
\end{thm}

\begin{proof}
 Observe that in both the case of a non-split torus and the case of a split torus, we are taking integrals over functions with compact support, so the derivatives with respect to $s$ commute with the integrals. Everything follows at once from applying the Leibniz rule to the formulas in Theorems \ref{thm_Zagier_anis} and \ref{thm_Zagier_split}.
\end{proof}


\section{Toroidal residues of Eisenstein series}
\label{section_residues}

\begin{pg}
 \label{def_residues}
 In this section, we prove that residues of Eisenstein series are not toroidal. Let $\chi\in\Xi$ with $\chi^2=\norm \ ^{\pm1}$, $f\in\cP(\chi)$, and $g\in G_\AA$. Then $E(g,f,s)$ as a function of $s$ has a pole at $s=0$, which is order of $1$. Thus the Eisenstein series has a nontrivial residue
 $$ R(g,f) \ := \ \Res_{s=0} \ E(g,f,s) \ =  \ \lim_{s\to0}\ s\cdot E(g,f,s), $$
 which is an automorphic form in $g$. Define 
 $$ R(\blanc,\chi)=\Res_{s=0}\ E(\blanc,\chi,s) $$
 if $\chi$ is unramified. The functional equation has a natural extension to residues of Eisenstein series. In particular, for unramified $\chi$, it becomes
 $$ R(\blanc,\chi)\ =\ \chi^2(\fc)\,R(\blanc,\chi^{-1}). $$
 Let $\varphi:\AA^2\to\CC$ be a Schwartz-Bruhat function. Then one can also define 
 $$ R(\blanc,\varphi,\chi)=\Res_{s=0}\ E(\blanc,\varphi,\chi). $$
 From the result for Eisenstein series, one obtains that for every $\varphi$, there is a $f\in\cP(\chi)$ such that $R(\blanc,\varphi,\chi)=R(\blanc,f)$, and vice versa.

 \label{residue=character_of_det}
 \label{residues_form_1-dim_representations}
 It turns out (see \cite[Thm.\ 4.19]{Gelbart-Jacquet}) that the residues are functions of a particular simple form. Let $\chi=\omega\norm \ ^{\pm1/2}$ be a quasi-character with $\omega^2=1$ and $f\in\cP(\chi)$, then
 $$R(g,f) \ = \ R(e,f)\cdot\omega\bigl(\det(g)\bigr)$$ 
 as functions of $g\in G_\AA$. This means that the $\cH$-submodule $\{R(\blanc,f)\}_{f\in\cP(\chi)}\subset\cA$ is $1$-dimensional.

 If $\chi^2=\norm \ ^{\pm1}$, we define for every $i\geq 0$
 $$ R^{(i)}(g,\varphi,\chi) \ = \ \lim_{s\to0}\ \frac{d^i}{ds^i}\ s\cdot E(g,\varphi,\chi,s), $$
 which defines an automorphic form in $g$.
\end{pg}

\begin{lemma}
 \label{toroidal_residues_anis}
 Let $T$ be a non-split torus corresponding to a separable quadratic field extension $E$ of $F$ and $\chi\in\Xi$ with $\chi^2=\norm\ ^{\pm1}$. For every Schwartz-Bruhat function $\varphi:\AA^2\to\CC$ and $g\in G_\AA$,
 $$ R_T(g,\varphi,\chi) \ = \ e_T(g,\varphi,\chi,0)\ \Res_{s=0}\ L_E(\chi\circ\N_{E/F},s+1/2) \;. $$
\end{lemma}

\begin{proof}
 With help of Theorem \ref{thm_Zagier_anis}, we calculate  
 \begin{align*}
  R_T(g,\varphi,\chi) &= \lim_{s\to0}\ s\ E_T(g,\varphi,\chi,s) \\
    &= \lim_{s\to0}\ s\ e_T(g,\varphi,\chi,s)\ L_E(\chi\circ\N_{E/F},s+1/2) \\
    &= e_T(g,\varphi,\chi,0)\ \Res_{s=0}\ L_E(\chi\circ\N_{E/F},s+1/2) \;.\mbox{\qedhere}
 \end{align*}
\end{proof}

\begin{lemma}
 \label{toroidal_residues_anis_vanish}
 Let $T$ be a non-split torus and $\chi=\omega\norm \ ^{\pm1/2}\in\Xi$ with $\omega^2=1$. There is a Schwartz-Bruhat function $\varphi$ such that $R_T(e,\varphi,\chi)\neq0$ if and only if $\omega=1$ or $\omega=\chi_T$.
\end{lemma}

\begin{proof}
 Observe that the residuum of 
 $$ L_E(\chi\circ\N_{E/F},s+1/2) \ = \ L_F(\omega,s+1/2\pm1/2)\ L_F(\omega\chi_T,s+1/2\pm1/2) $$
 at $s=0$ is nontrivial if and only if one of the two factors is the zeta function of $F$, and this happens if $\omega=1$ or $\omega=\chi_T^{-1}=\chi_T$.

 If $\Res_{s=0} L_E(\chi\circ\N_{E/F},s+1/2)=0$, then $R_T(e,\varphi,\chi)=0$ for all Schwartz-Bruhat functions $\varphi$ by Lemma \ref{toroidal_residues_anis}. If not, then $R(e,\varphi_T,\chi)=1\cdot\,\Res_{s=0} L_E(\chi\circ\N_{E/F},s+1/2)$ (Theorem \ref{thm_Zagier_anis}\,\eqref{anis2}) does not vanish.
\end{proof}

\begin{lemma}
 \label{toroidal_derivatives_of_residues_anis}
 Let $T$ be a non-split torus and $\chi\in\Xi$ with $\chi^2=\norm\ ^{\pm1}$.
 If $R_T(e,\varphi,\chi)=0$ for all Schwartz-Bruhat functions $\varphi$, 
 then there exists a Schwartz-Bruhat function $\varphi$ such that $R_T^{(1)}(e,\varphi,\chi) \ \neq \ 0$.
\end{lemma}

\begin{proof}
 By Lemma \ref{toroidal_residues_anis_vanish}, we have that $R_T(g,\varphi,\chi)=0$ for all $\varphi$ and $g\in G_\AA$ if and only if 
 $L_E(\omega\circ\N_{E/F},\blanc)$ has no pole at $0$ or $1$.
 With the help of Theorem \ref{thm_Zagier_anis}, we calculate
 \begin{align*}
  R_T^{(1)}(e,\varphi_T,\chi) 
  &= \lim_{s\to0}\ \frac{d}{ds}\ s\ E_T(e,\varphi_T,\chi,s) \\
  &= \lim_{s\to0}\ \frac{d}{ds}\ s\ e_T(e,\varphi_T,\chi,s)\,L_E(\chi\circ\N_{E/F},s+1/2) \\
  &= \lim_{s\to0}\ \biggl(e_T(e,\varphi_T,\chi,s)\,L_E(\chi\circ\N_{E/F},s+1/2) \\ 
  &\hfill\hspace{3cm} +\ s\ \frac{d}{ds}\ e_T(e,\varphi_T,\chi,s)\,L_E(\chi\circ\N_{E/F},s+1/2) \biggr) \\
  &= e_T(e,\varphi_T,\chi,0)\,L_E(\omega\circ\N_{E/F},1/2\pm1/2)\;,
 \end{align*}
 which does not vanish by Theorem \ref{thm_Zagier_anis} \eqref{anis2} and by the non-vanishing of $L$-functions of non-trivial finite Hecke character in $0$ 1nd $1$ (cf.\ paragraph \ref{L_series_has_no_zero_in_0,1}).
\end{proof}

\begin{pg}
 Let $T$ be a split torus and $\chi=\omega\norm \ ^{\pm1/2}\in\Xi$ with $\omega^2=1$.
 Let $N$ be the unipotent radical of a Borel subgroup $B\subset G$.
 Then
 $$ (\omega\circ\det)_N(g) \ = \hspace{-2pt}\int \limits_{\lquot{N_F}{N_\AA}} \hspace{-2pt} \omega\circ\det(n\!g)\,d\!n 
    \ = \ \omega\circ\det(g) \;. $$
 Consequently 
 $$ R_T(e,\varphi,\chi) \ = \ 0 $$
 for every Schwartz-Bruhat function $\varphi$ by the definition of the toroidal integral for a split torus.

 We summarise:
\end{pg}

\begin{thm}
 \label{E-toroidal_residues}
 Let $E$ be a separable quadratic algebra extension of $F$, $\chi_E$ the character from paragraph \ref{def_chi_T} and
 $\chi=\omega\norm \ ^{\pm1/2}\in\Xi$ with $\omega^2=1$. 
 \begin{enumerate}
  \item If $\omega$ is trivial, then $R(\blanc,\chi)\in\cA_\tor(E)$ if and only if $E\simeq F\oplus F$.
  \item If $\omega$ is nontrivial, then $R(\blanc,\chi)\in\cA_\tor(E)$ if and only if $\omega\neq\chi_E$.
  \item If $E$ is a field and $n\geq1$, then $R^{(n)}(\blanc,\chi)\notin\cA_\tor(E)$.\qed
 \end{enumerate} 
\end{thm}

 Since by class field theory every quadratic character of $\AA^\times/F^\times$ is of the form $\chi_E$ for some quadratic separable field extension $E$ of $F$, we finally obtain:

\begin{thm}
 \label{no_toroidal_residues}
 $\cR_\tor=\{0\}$.\qed
\end{thm}


\section{Non-vanishing for quadratic twists of $L$-series}
\label{section_non-vanishing}

\noindent
 Let $q$ be odd throughout this section. The goal of this section is to show that not all quadratic twists $L(\chi\chi_E,s)$ vanish simultaneous in a given $\chi$ and $s$. For the proof of this result, we employ Double Dirichlet series, which are certain weighted sums over all quadratic twists of $L(\chi,s)$. Essentially, we show that this series does not vanish for any $\chi$ and $s$ (as a function in another parameter $w$), which implies that at least one quadratic twist of the $L$-series in question is not zero in $\chi$ and $s$. 

\begin{thm}
 \label{thm_non-vanish}
 Let $q$ be odd. Then there is for every $\chi\in\Xi$ and $s\in\CC$ a separable quadratic field extension $E$ of $F$ such that $L(\chi\chi_E,s)\neq0$.
\end{thm}

\noindent
This theorem together with Theorem \ref{thm_Zagier_split} implies that the only (derivatives of) Eisenstein series that are toroidal are those that correspond to zeros of the (untwisted) $L$-series $L(\chi,s+1/2)$. More precisely:

\begin{thm}
 \label{toroidal_if_and_only_L-series=0}
 Let $q$ be odd. Let $\chi\in\Xi$, $s\in\CC$ and $n\geq0$. Then $E^{(n-1)}(\blanc,\varphi,\chi)$ is toroidal for every $\varphi$ if and only if $L(\chi,s)$ vanishes in $s=1/2$ to order at least $n$.\qed
\end{thm}

\begin{pg}
 The rest of this section is devoted to prove Theorem \ref{thm_non-vanish}, with the exception of some conculding remarks in the end. Since the proof is similar to the analogous statement for number fields (see \cite{Cornelissen-Lorscheid2}), we use the language of ideal classes---in contrast to the paper \cite{Fisher-Friedberg2} where Double Dirichlet series are introduced for function fields; namely, the latter paper the language of divisors on the smooth projective curve corresponding to the function field $F$ is used. All background for Double Dirichlet series in the function field case, however, can be found in \cite{Fisher-Friedberg2}. We review the definition of Double Dirichlet series.

 One difficulty in the definition for general global fields is that they may have non-trivial class number, which implies that the quadratic residue symbol is not defined for every ideal in a ring of integers. This lack can be circumvented in the following way. Let $S\subset X$ be a non-empty finite set of places such that $\cO_S=\{a\in F|\ord_x(a)\geq0\;\forall x\notin S\}$ has class number $1$. Let $C$ be the formal sum of all places $x\in S$ and let $H_C=\AA^\times/F^\times \cO_\AA^\times(C)$ be the ray class group of modulus $C$, where $\cO_\AA^\times(C)$ is the subgroup of all $(a_x)\in\cO_\AA^\times$ such that $\ord(1-a_x)\geq1$ for all $x\in S$. Let $R_C=H_C\otimes\ZZ/2\ZZ$, which is a finite group. Choose a minimal set $b_1,\dots b_r$ of generators for $R_C$ and choose a set $\cE_0$ of ideals in $\cO_S$ that represent these generators. For every $E_0$ in $\cE_0$, let $m_{E_0}$ be an element of $\cO_S^\times$ that generates the (principal) ideal $E_0$ in $\cO_S$. Let $\cE$ denote a set of representatives of $R_C$ of the form $E=\prod E_0^{n_{E_0}}$, where the $E_0$ are elements of $\cE_0$ and the $n_{E_0}$ are natural numbers, and set accordingly $m_E=\prod m_{E_0}^{n_{E_0}}$ with the convention that the trivial element of $R_C$ is represented by $\cO_S$ and that $m_{\cO_S}=1$. Note that $m_E$ generates $E$. 

 Let $I(S)$ denote the set of fractional ideals of $\cO_S$. For $d\in I(S)$, write $d=(a)EG^2$ with $E\in\cE$, $G\in I(S)$ and $a\in F^\times$ with $a\equiv 1\ (\textup{mod}\ C)$. Define $\chi_d = \chi_E $ for the quadratic field extension $E=F[\sqrt{am_E}]$ of $F$, which should be thought of as the quadratic residue symbol for $d$. This definition does not depend on the decomposition $d=(a)EG^2$ of $d$ (\cite[Lemma 1.1]{Fisher-Friedberg2})---it only depends on the choices of $S$ and the $m_E$.
\end{pg}

\begin{pg}
 \label{residuum_of_DD}
 Let $J(S)\subset I(S)$ be the set of (integral) ideals of $\cO_S$. Let $\chi\in\Xi$ be a character, i.e.\ $\Re\chi=0$, and $s\in\CC$. We define the weight factor to be
 $$ a(\chi,s,d)=\sum_{\substack{e_1,e_2\in J(S)\\(e_1e_2)^2\,|\,d}}\frac{\mu(e_1)\chi_d(e_1)\chi(e_1e_2^2)}{\norm{e_1}^s\norm{e_2}^{2s-1}}  $$
 where $\mu$ is the M\"obius function. For $d\in I(S)$, denote by $S_d$ the set of primes above $d$, and let $L_{S\cup S_d}(\chi\chi_d,s)$ be the $L$-series with the factors for primes in $S\cup S_d$ removed. Let $\rho\in\Xi$ be a character unramified outside $S$ and $w\in\CC$. The \emph{Double Dirichlet series of $\chi$ and $\rho$ in $s$ and $w$} is the series
 $$ Z(s,w;\chi,\rho) = \sum_{d\in J(S)} \frac{L_{S\cup S_d}(\chi\chi_d,s)\rho(d)}{\norm d ^w} \cdot a(\chi,s,d). $$
 This expression is absolute convergent for $\Re s>1$ and $\Re w>1$ and has a meromorphic continuation to all $s$ and $w$ (\cite[Thm.\ 4.1]{Fisher-Friedberg2}). It has a pole in $w=1$ (independent of $s$). The residuum of $Z(s,w;\chi,\rho)$ at $w=1$ can be calculated precisely as in \cite[Section 4]{Friedberg-Hoffstein-Lieman}, which treats the case of Double Dirichlet series for $n$-th order twists with $n\geq3$. Namely,
 $$ \Res_{w=1} Z(s,w;\chi,\rho) = \left\{\begin{array}{ll}\Res_{w=1}\zeta_F(w)\cdot\prod_{x\in S}\zeta_{F,x}(1)^{-1}\cdot L_S(\chi^2,2s) & \text{if }\rho=1, \\ 0 & \text{if }\rho\neq 1.\end{array}        \right. $$
\end{pg}
 
\begin{pg}
 We describe two more series, which we shall need in order to proof the theorem. Since the characteristic function $\delta_0$ of $0\in R_C$ equals $(\#R_C)^{-1}\sum_{\rho\in\widehat{R_C}}\rho$, we have that
 $$ \sum_{\substack{d\in J(S)\\d\text{ principal}}} \frac{L_{S\cup S_d}(\chi\chi_d,s)}{\norm d ^w}\cdot a(\chi,s,d) \quad = \quad \frac 1{\#R_C}\sum_{\rho\in\widehat{R_C}}Z(s,w;\chi,\rho), $$
 is a meromorphic function in $s$ and $w$, which we denote by $Z^0(s,w;\chi)$. This series converges for every $s$ if $\Re w$ is large enough. This is because of the Phragm\'en-Lindel\"of estimates in the $d$-aspect of the form
 $$ \norm{L_{S\cup S_d}(\chi,s) a(\chi,s,d^2)} \ll \norm d $$
 (cf.\ \cite[3.3]{Chinta-Friedberg-Hoffstein}). Since $q$ is odd, every square $d\in(F^\times)^2$ has two different roots. Note that $S_{d^2}=S_d$ and that $\chi_{d^2}$ is trivial. Thus 
 $$ \sum_{\substack{d\in J(S)\\d\in (F^\times)^2\text{ principal}}} \frac{L_{S\cup S_d}(\chi\chi_d,s)}{\norm d ^w}\cdot a(\chi,s,d) \quad = \quad \frac 12\sum_{\substack{d\in J(S)\\d \text{ principal}}} \frac{L_{S\cup S_d}(\chi,s)}{\norm d ^{2w}}\cdot a(\chi,s,d^2), $$
 a series, which we denote by $Z^0_\textup{sq}(s,w;\chi)$. It converges for $\Re w > 1/2$ for the same reasons as for $Z^0(s,w;\chi)$.
\end{pg}

\begin{pg}
 We can proceed with the proof of Theorem \ref{thm_non-vanish}. Since $L(\chi,s)=L(\chi\norm\ ^{-\Re\chi},s+\Re\chi)$, we may assume that $\chi$ is a character in order to prove the theorem. By the functional equation of $L(\chi,s)$, we may further assume that $\Re s\geq1/2$. We fix $\chi\in\Xi$ and $s$. The twisted $L$-series occuring in the series
 $$ \sum_{\substack{d\in J(S)\\d\text{ principal}\\ \text{not a square}}} \frac{L_{S\cup S_d}(\chi\chi_d,s)}{\norm d ^w}\cdot a(\chi,s,d) \quad = \quad Z^0(s,w;\chi) - Z^0_\textup{sq}(s,w;\chi)  $$
 are all of the form $L(\chi\chi_E,s)$ (up to some non-vanishing factors for the places in $S\cup S_d$) for a non-trivial quadratic character $\chi_E$ which corresponds to a separable quadratic field extension $E$ of $F$. If we can show that this series does not vanish (as a function in $w$), then at least one of the terms $L(\chi\chi_E,s)$ is not zero, and the theorem follows.

 We do this by showing that $Z^0(s,w;\chi) - Z^0_\textup{sq}(s,w;\chi)$ has a non-trivial residue in $w=1$. Since the defining series for $Z^0_\textup{sq}(s,w;\chi)$ converges in $w=1$,
 $$ \Res_{w=1} \bigl(Z^0(s,w;\chi) - Z^0_\textup{sq}(s,w;\chi)\bigr) = \Res_{w=1} Z^0(s,w;\chi) =  \frac {1}{\#R_C}\sum_{\rho\in\widehat{R_C}}\Res_{w=1} Z(s,w;\chi,\rho), $$
 which is a sum over $0$'s except for the summand corresponding to the trivial character $\rho=1$, for which we obtain a term of the form $c\cdot L_S(\chi^2,2s)$ for a non-zero constant $c$, as explained in paragraph \ref{residuum_of_DD}. As the real part of $s$ was assumed to be at least $1/2$, neither $L_S(\chi^2,2s)$ is zero. This accomplishes the proof of the theorem. \qed
\end{pg}

\begin{rem}
 \label{rem_omit_E}
 Note that the statement of Theorem \ref{thm_non-vanish} can be strengthened in the form that for given $\chi$ and $s$, the term $L(\chi\chi_E,s)$ is not zero for infinitely many $E$. This is because one can substract a similar term to $Z^0_\textup{sq}(s,w;\chi)$ for the $\chi_E$-twists of a given quadratic field extension $E$ from $Z^0(s,w;\chi)$, which does not change the residuum. In particular, for odd characteristic, we could omit finitely many separable quadratic algebra extensions $E$ in the intersection $\cA_\tor=\bigcap\cA_\tor(E)$ without changing $\cA_\tor$.
\end{rem}

\begin{rem}
 As the formalism of Double Dirichlet series does not apply to characteristic $2$, the above proof does not say something about this case. Note that also in Ulmer's paper \cite{Ulmer} characteristic $2$ is excluded. Namely, Theorems 1.1 and 5.2 of \cite{Ulmer} imply that for every global function field $F$ of characteristic different from $2$, there is integer $n_0$ such that for every $n\geq n_0$ the quadratic twists $L(\chi\chi_E,s)$ do not vanish simultaneously in $\chi$ and $s$ for the constant field extension $\FF_{q^n}F$. 

 For low genus $g\leq1$ and unramified $\chi$, we can, however, look at certain explicit extensions to exclude a common zero (unless the class number $h$ equals $q+1$). For genus $g=0$, all unramified quasi-characters are principal, i.e.\ of the form $\norm\ ^s$. It suffices thus to consider the zeta funtion of $F$, which is $\zeta_F(s)=\frac{1}{(1-q^s)(1-q^{1-s})}$. It has no zero, and thus there is no unramified $F\oplus F$-toroidal Eisenstein series (notice the non-vanishing result for the split torus in the unramified case in Theorem \ref{thm_Zagier_split}). For genus $g=1$, the consideration of all unramified extensions of $F$ yields the desired result as long as $h\neq q+1$, see \cite[Cor.\ 7.13]{Lorscheid3}.
\end{rem}

\bigskip

\part{Toroidal representations}
\label{part_representations}

\noindent 
A \emph{toroidal representation} is a subrepresentation of $\cA_\tor$. In this part, we investigate properties of toroidal representation like temperedness and admissibility, the former notion being closely related to the Riemann hypothesis for function fields as explained in the following section.


\section{Connection with the Riemann hypothesis}
\label{section_riemann}


\begin{pg}
 \label{tensor_product_of_repr}
 In this section, we translate the observation of Don Zagier (\cite[pp. 295--296]{Zagier1}) that unitarizability of the space of toroidal automorphic forms implies the Riemann hypothesis to the setting of global function fields.

 We begin wth recalling some background in automorphic representations, by which we mean subrepresentations of $\cA$. Every (infinite-dimensional) irreducible automorphic representation $V$ decomposes into a restricted tensor product $\bigotimes'_{x\in X}\cP(\chi_x)$ of the principal series representations $\cP(\chi_x)$ of $G(F_x)$ (cf.\ \cite{Bump} for details on restricted tensor products). The principal series $\cP(\chi_x)$ of the quasi-character $\chi_x$ is the space of smooth functions $f$ on $G(F_x)$ that satisfy 
 $$ f\Bigl(\smallmat a b {} d g \Bigr) \ = \ \norm{\frac ad}_x^{1/2}\chi_x\Bigl(\frac ad\Bigr)\,f(g) $$
 for all $\tinymat a b {} d $ and $g$ in $G(F_x)$. If $V\simeq\cP(\chi)$, then the local characters $\chi_x$ are the restrictions of $\chi$ to $F_x$. An irreducible subrepresentation $V$ is called {\it tempered} if it is isomorphic to $\bigotimes'\cP(\chi_x)$ where all $\chi_x$ are characters.

 Let the {\it Eisenstein part $\cE$} be the vector space spanned by all Eisenstein series and their derivatives, the {\it residual part} $\cR$ be the vector space spanned by the residues of Eisenstein series and their derivatives in the sense of paragraph \ref{def_derivatives}, and the {\it cuspidal part} $\cA_0$ be the space of cusp forms (see paragraph \ref{cusp_forms}). We shall refer to $\varcE:=\cE\oplus\cR$ as the {\it completed Eisenstein part}. A theorem of Waldspurger and Moeglin in \cite{Moeglin-Waldspurger} says that 
 $$ \cA \ = \ \cA_0 \, \oplus \, \cE \, \oplus\, \cR $$
 as automorphic representation.
\end{pg}

\begin{pg}
 \label{def_lambda_x}
 Let $K$ be the standard maximal compact open subgroup of $G_\AA$ (see paragraph \ref{def_K}). Then we denote by $\cH_K$ the subalgebra of bi-$K$-invariant functions of $\cH$, i.e.\ the algebra of locally constant functions $\Phi:G_\AA\to\CC$ with compact support such that $\Phi(kgk')=\Phi(g)$ for all $k,k'\in K$. We define for every $x\in X$ the Hecke operator $\Phi_x\in\cH_K$ as the characteristic function of $K\tinymat \pi_x {} {} 1 K$ and the Hecke operator $\Psi_x$ as the characteristic function of $K\tinymat \pi_x {} {} \pi_x K$. Then $\cH_K=\CC[\Phi_x,\Psi_x^{\pm1}]_{x\in X}$, in particular $\cH_K$ is commutative. Note that $\Psi_x$ and $\Psi_x^{-1}$ operate trivial on the space $\cA^K$ of $K$-invariant automorphic forms. We denote by $\cA^\nr$ the smallest subrepresentation of $\cA$ that contains $\cA^K$.
 
 Let $E(\blanc,\chi)$ be the unramified Eisenstein series associated to $\chi$ (see paragraph \ref{def_unram_Eisenstein}) and $q_x=q^{\deg x}$. Then 
 $$\Phi_x.E(\blanc,\chi) \ = \ q_x^{1/2}\bigl(\chi^{-1}(\pi_x) + \chi(\pi_x)\bigr)E(\blanc,\chi) $$
 (cf.\ \cite[\S3 Lemma 3.7]{Gelbart1}), thus $E(\blanc,\chi)$ is an eigenfunction of $\Phi_x$ with eigenvalue $\lambda_x(\chi)=q_x^{1/2}(\chi^{-1}(\pi_x) + \chi(\pi_x))$.
\end{pg}

\begin{lemma}
 \label{eigenvalues_of_tempered_characters}
 Let $\chi\in\Xi_0$, then the following are equivalent.
 \begin{enumerate}
  \item \label{temp1} $\cP(\chi)$ is a tempered representation.
  \item \label{temp2} $\Re\chi=0$.
  \item \label{temp3} $\lambda_x(\chi)\in[-2q_x^{1/2},2q_x^{1/2}]$ for all $x\in X$.
  \item \label{temp4} $\lambda_x(\chi)\in[-2q_x^{1/2},2q_x^{1/2}]$ for one $x\in X$.
 \end{enumerate}
\end{lemma}

\begin{proof}
 We have $\cP(\chi)\simeq\bigotimes'\cP(\chi_x)$ with $\chi_x=\chi\vert_{F_x}$ (see paragraph \ref{tensor_product_of_repr}). All $\chi_x$ are characters if and only if $\chi$ is a character. This is the case if $\im\chi\subset\SS^1$, or equivalently if $\Re\chi=0$. Thus the equivalence of \eqref{temp1} and \eqref{temp2}.

 Assume \eqref{temp2}. Then $\im\chi\subset\SS^1$, and $ \lambda_x(\chi) = q_x^{1/2}\bigr(\chi^{-1}(\pi_x)+\chi(\pi_x)\bigr)$ for every $x\in  X$. But $\chi^{-1}(\pi_x)$ is the complex conjugate of $\chi(\pi_x)$, therefore $\chi^{-1}(\pi_x)+\chi(\pi_x)\in[-2,2]$; thus \eqref{temp3}. The implication from \eqref{temp3} to \eqref{temp4} is trivial.
 
 Conversely, $\chi^{-1}(\pi_x)+\chi(\pi_x)\in[-2,2]$ only if $\chi^{-1}(\pi_x)$ is the complex conjugate of $\chi(\pi_x)$, thus $\chi(\pi_x)\in\SS^1$. Since $\AA^\times/\bigl(F^\times\cO_\AA^\times\langle\pi_x\rangle\bigr))$ is a finite group, all values of $\chi$ are contained in $\SS^1$ and $\Re\chi=0$; thus \eqref{temp2}.
\end{proof}  

\begin{thm}[Zagier]
 \label{thm_zagier}
 If every irreducible subrepresentation of $\cA_\tor^\nr$ is a tempered representation, then all zeros of $\zeta_F$ have real part $1/2$. If furthermore, $\cA_\tor^\nr$ is itself a tempered representation, then $\zeta_F$ has only simple zeros.
\end{thm}
 
\begin{proof}
 By Theorem \ref{lemma_zagier_for_derivatives}, we know that every zero $1/2+s$ of order $n$ of $\zeta_F$ yields that all the functions $E(\blanc,\norm \ ^s),\dotsc,E^{(n-1)}(\blanc,\norm \ ^s)$ are toroidal. It is well-known that only the zeroth derivative $E(\blanc,\norm \ ^s)$ generates an irreducible representation (cf.\ Section \ref{section_basis} for more detail). If this representation is tempered, then the real part of $s$ is $0$ by Lemma \ref{eigenvalues_of_tempered_characters}.
 
 If furthermore $\cA_\tor^\nr$ is the direct sum of irreducible tempered subrepresentations, then no proper derivative of an Eisenstein series can occur and the zeros of $\zeta_F$ must be of order $1$.
\end{proof} 
 
 By Lemma \ref{eigenvalues_of_tempered_characters}, we obtain:

\begin{cor}
 \label{cor_zagier}
 If there is a place $x$ such that the eigenvalue of every $\Phi_x$-eigenfunction in $\cA_\tor^K$ lies in the interval $[-2q_x^{1/2},2q_x^{1/2}]$, then all zeros of $\zeta_F$ have real part $1/2$. \qed
\end{cor}

\begin{rem}
 By means of this corollary, it is possible to verify the Riemann hypothesis for the function field $F$ in certain cases. For the proofs for rational function fields $\FF_q(T)$ and elliptic function fields with class number $1$, see \cite{Cornelissen-Lorscheid}. For other elliptic function fields, it is possible to tackle this problem by the theory of graphs of Hecke operators, but by direct calculations it could only be shown that the irreducible subrepresentations of $\cA_\tor^\nr$ are unitarizable (see \cite[section 8.4]{Lorscheid-thesis}).
\end{rem}


\section{Temperedness and admissibility of toroidal representations}
\label{section_tempered}

\noindent
 The implication of Theorem \ref{thm_zagier} is of hypothetical nature as the Riemann hypothesis is proven for global function fields, which is known as the Hasse-Weil theorem. Conversely, we can make use of the Hasse-Weil theorem to prove that every irreducible suquotient of $\cA_\tor$ is tempered. We further conclude that $\cA_\tor$ is an admissible representation.

\begin{thm}
 \label{thm_tempered}
 Every irreducible subquotient of $\cA_\tor$ is tempered. 
\end{thm}

\begin{proof}
 The space $\cA_\tor$ inherits the decomposition of $\cA$ into a cuspidal, an Eisenstein and a residual part, thus we can investigate these parts, with name, $\cA_{0,\tor}$, $\cE_\tor$ and $\cR_\tor$, separately.

 The Ramanujan-Petersson conjecture claims that all irreducible cuspidal representations for $\GL(2)$ are tempered. This conjecture was proven by Drinfel\cprime d in the function field case (\cite{Drinfeld2}), which implies the corresponding statement for $\cA_{0,\tor}$.

 Recall from paragraph \ref{embedding_principal_series_to_automorphic_forms} that $E(\blanc,\varphi,\chi)$ generates a subrepresentation of $\cA$ that is isomorphic to $\cP(\chi)$. Furthermore, if $V\subset\cA$ is generated by $E(\blanc,\varphi,\chi),\dotsc,E^{(n)}(\blanc,\varphi,\chi)$ as $G_\AA$-module, and $V'\subset\cA$ is generated by $E(\blanc,\varphi,\chi),\dotsc,E^{(n-1)}(\blanc,\varphi,\chi)$, then also the quotient representation $\rquot{V}{V'}$ is isomorphic to $\cP(\chi)$ (or trivial). Thus the isomorphism types of all irreducible subquotients of $\cE_\tor^\nr$ are determined by the irreducible subrepresentations of $\cE_\tor^\nr$ and it suffices to investigate the irreducible subrepresentations of $\cE_\tor(E)$.

 By Theorem \ref{thm_Zagier_anis}, a non-trivial Eisenstein series $E(\blanc,\varphi,\chi)$ is toroidal only if $L(\chi,1/2)=0$ or $L(\chi\chi',1/2)=0$ for some quadratic character $\chi'$. By the Hasse-Weil theorem, $\Re\chi=1/2$ (note that $\Re\chi'=0$ for quadratic $\chi'$), and $E(\blanc,\varphi,\chi)$ generates a tempered representation.

 To conclude the proof, we observe that $\cR_\tor=0$ by Theorem \ref{no_toroidal_residues}.
\end{proof}

\begin{thm}
 \label{thm_admissible}
 For every separable quadratic field extension $E$ of $F$, the representation $\cA_\tor(E)$ is admissible. Consequently, $\cA_\tor$ is an admissible representation.
\end{thm}

\begin{proof}
 We have to show that for every compact open subgroup $K'$ of $K$, the complex vector space $(\cA_\tor(E))^{K'}$ is finite dimensional. Due to the decomposition $\cA_\tor(E)=\cA_{0,\tor}(E)\oplus\cE_\tor(E)\oplus\cR_\tor(E)$, we can verify this condition for each of the summands separately. Since $\cA_0$ and $\cR$ are admissible, the subrepresentations $\cA_{0,\tor}(E)$ and $\cR_\tor(E)$ are so, too.

 The admissibility of $\cE_\tor(E)$ can be seen as follows. Note that $\chi$ is trivial on $N\cO_\AA$ (where $N\in\cO_\AA$ is the ramification) if and only if $E(\blanc,\varphi,\chi)\in\cE^{K_N}$ for $K_N=\{g\in K|g\equiv e \ (\text{mod}\ N)\}$. Every compact open set $K'$ is contained in $K_N$ for some $N$. So it suffices to prove finite-dimensionality only for the spaces of the form $\cE^{K_N}$.

 Every subrepresentation $V$ that contains $E^{(i)}(\blanc,\varphi,\chi)$ for some $\varphi$ and some $\chi$, contains also $E^{(j)}(\blanc,\varphi,\chi)$ for all $j<i$. This means that if a derivative of an Eisenstein series is toroidal, then all derivatives of lower order are also toroidal. Since $L$-series have only finitely many zeros (with finite multiplicities) and by the non-vanishing result for $e_T^{(i)}(g,\varphi,\chi)$ (Theorem \ref{thm_Zagier_anis} \eqref{anis2}), the product $e_T(g,\varphi,\chi)\ L(\chi,1/2)\ L(\chi\chi_T,1/2)$ vanishes only for finitely many $\chi$ that are trivial on $N\cO_\AA$, and it vanishes only with finite multiplicity. 

 Let $\chi$ be a zero of multiplicity $n$ of this product. Then $E(\blanc,\varphi,\chi),\dotsc,E^{(n-1)}(\blanc,\varphi,\chi)$ are toroidal by Theorem \ref{lemma_zagier_for_derivatives}, but $E^{(n)}(\blanc,\varphi,\chi)$ is not. Each of the $E^{(i)}(\blanc,\varphi,\chi)$ corresponds to an irreducible subquotient of $\cE_\tor(E)$ and each irreducible subquotient of $\cE_\tor(E)$ is of this form. Consequently, $\cE_\tor(E)$ is admissible.
\end{proof}

\bigskip

\part{The space of unramified toroidal automorphic forms}
\label{part_dimension}

\noindent
In this part of the paper, we show that the dimension of $\cE_\tor^K$ equals $(g-1)h+1$ for odd characteristics where $g$ is the genus and $h$ is the class number of $F$. To carry out the proof of this dimension formula, we first have to establish a basis for $\cA^K$ that allows us to determine the dimension of $\cE_\tor^K$. In the second section of this part, we will compare this dimensions with the number of zeros of $L(\chi,1/2)$ (with multiplicity) in unramified characters $\chi$.


\section{A basis for the space of unramified automorphic forms}
\label{section_basis}

\begin{pg}
 \label{Eisenstein_part} \label{thm_Li}
 We fix some terminology. For $\lambda\in\CC$, and $\Phi\in\cH_K$, define the {\it space of $\Phi$-eigenfunctions with eigenvalue $\lambda$} as
 $$ \cA(\Phi,\lambda) \ = \ \{f\in\cA\mid\Phi(f)=\lambda f \} \;, $$
 and for a subrepresentation $V\subset\cA$, define $V(\Phi,\lambda)=V\cap\cA(\Phi,\lambda)$. By an $\cH_K$-eigenfunction, we mean a simultaneous eigenfunction for all $\Phi\in\cH_K$. Notice that it suffices to consider the action of the Hecke operators $\Phi_x$ to determine the action of $\cH_K$ since $\cH_K$ is generated as an algebra by $\Phi_x$ and operators that act trivial on $\cA$ (see paragraph \ref{def_lambda_x}). Note that $\cH_k$ acts on $\cA^K$.

 All spaces $\cA(\Phi,\lambda)$ inherit the decomposition of $\cA$ into a cuspidal, an Eisenstein and a residual part. The unramified Eisenstein series $E(\blanc,\chi)$ are $\Phi_x$-eigenfunctions with eigenvalue $\lambda_x(\chi)=q_x^{1/2}(\chi^{-1}(\pi_x) + \chi(\pi_x))$ unless $\chi^2=\norm\ ^{\pm1}$, in which case the residue $R(\blanc,\chi)$ is an $\Phi_x$-eigenfunction with eigenvalue $\lambda_x(\chi)$ (cf.\ paragraph \ref{def_lambda_x}). Note that none of these functions is trivial and the only linear dependencies between these functions are given by the functional equations $E(\blanc,\chi)= \chi^2(\fc)E(\blanc,\chi^{-1})$ resp.\ $R(\blanc,\chi)= \chi^2(\fc)R(\blanc,\chi^{-1})$ (cf.\ paragraphs \ref{functional_equation_eisenstein} and \ref{def_residues}). These functions are $\cH_K$-eigenfunctions and generate the $K$-invariant part $\widetilde\cE^K$ of the generalised Eisenstein part $\widetilde\cE=\cE\oplus\cR$.

 The Jordan decomposition implies that for every $x\in X$, the Hecke operator $\Phi_x$ decomposes $\cA^K$ into a direct sum of subspaces that are the (typically infinite-dimensional) generalised eigenspaces, on which $\Phi_x$ operates as a Jordan block in an appropriate basis. Since all the operators $\Phi_x$ commute, these generalised eigenspaces coincide for the various $\Phi_x$.

 \label{def_lambda_x(chi)}
 Denote the derivatives of $E(\blanc,\chi)$ in the sense of paragraph \ref{def_derivatives} by $E^{(i)}(\blanc,\chi)$. Define for all $\chi\in\Xi_0$, $x\in X$ and $l\geq0$ the value
 $$ \lambda_x^{(l)}(\chi) \ := \ q_x^{1/2}\bigl(\chi^{-1}(\pi_x) + (-1)^{l}\chi(\pi_x)\bigr) \;. $$
 Note that $\lambda_x^{(l)}(\chi)$ only depends on the parity of $l$. Put $\lambda_x(\chi)=\lambda_x^{(l)}(\chi)$ if $l$ is even and $\lambda_x^-(\chi)=\lambda_x^{(l)}(\chi)$ if $l$ is odd.
\end{pg}

\begin{lemma}
 \label{gen_Eisenstein_eigenvalue}
 If $\chi\in\Xi_0$ with $\chi^2\neq\norm \ ^{\pm1}$, then for every $x\in X$,
 $$ \Phi_x E^{(i)}(g,\chi) \ = \ \sum_{k=0}^i \binom ik \bigl(\ln q_x\bigr)^{i-k} \, \lambda_x^{(i-k)}(\chi) \, E^{(k)}(g,\chi) \;. $$
\end{lemma}

\begin{proof}
 Observe that
 $$ \frac d{ds} \lambda_x^{(l)}(\chi\norm \ ^s) \ = \ (\ln q_x)\, \lambda_x^{(l+1)}(\chi\norm \ ^s) \;. $$
 The formula is obtained by taking derivatives on both sides of the functional equation for Eisenstein series (cf.\ paragraph \ref{functional_equation_eisenstein}) and applying the Leibniz rule to the right hand side.
\end{proof}

\begin{lemma}
 \label{lemma_chi^2=0_and_lambda_x^-(chi)=0}
 Let $\chi\in\Xi_0$. Then $\chi^2=1$ if and only if $\lambda_x^-(\chi)$ vanishes for all places $x$.
\end{lemma}

\begin{proof}
 Observe that for every $\pi_x$, we have
 $$ q_x^{-1/2}\lambda_x^-(\chi)=\chi^{-1}(\pi_x) - \chi(\pi_x)=0 \ \ \iff\ \ \chi(\pi_x) = \chi^{-1}(\pi_x)
    \ \ \iff \ \ \chi^2(\pi_x) = 1 \;. $$
 Since the $\pi_x$'s generate $\lrquot{F^\times}{\AA^\times}{\cO_\AA^\times}$, the quasi-character $\chi^2$ is determined by its values on the $\pi_x$'s.
\end{proof}

\begin{prop}
 \label{series_of_Eisenstein_derivatives}
 Let $\chi\in\Xi_0$ with $\chi^2\notin\{1,\norm \ ^{\pm1}\}$. Then
 $$ \bigl\{ E(\blanc,\chi),\ E^{(1)}(\blanc,\chi),\ E^{(2)}(\blanc,\chi),\ldots\bigr\} $$
 is linearly independent and spans a vector space on which $\cH_K$ acts. In particular none of these functions vanishes.
\end{prop}

\begin{proof}
 By Lemma \ref{gen_Eisenstein_eigenvalue}, it is clear that the span of the functions is an $\cH_K$-module. We do induction on $n=\# \bigl\{ E(\blanc,\chi), E^{(1)}(\blanc,\chi),\ldots,E^{(n-1)}(\blanc,\chi)\bigr\}$.
 
 Since $E(\blanc,\chi)$ is not zero, the case $n=1$ follows. For $n>1$, assume that there exists a relation
 $$ E^{(n)}(\blanc,\chi) \ = \ c_{n-1} E^{(n-1)}(\blanc,\chi)\, + \ldots +\, c_0 E(\blanc,\chi) \;. $$
 We derive a contradiction as follows. For every place $x$, we have on the one hand,
 \begin{align*}
  \Phi_x\,E^{(n)}(\blanc,\chi) & \hspace{3,7pt} = \hspace{3,7pt} c_{n-1}\,\Phi_x\,E^{(n-1)}(\blanc,\chi) + \ldots + c_0\,\Phi_x\,E(\blanc,\chi)\\
  & \underset{\ref{gen_Eisenstein_eigenvalue}}{=} c_{n-1}\,\lambda_x(\chi)\,E^{(n-1)}(\blanc,\chi) \ + \ \bigl(\text{terms in lower derivatives of }E(\blanc,\chi)\bigr) \;, 
 \end{align*}
 and on the other hand,
 \begin{align*}
  \Phi_x\,E^{(n)}(\blanc,\chi) & \underset{\ref{gen_Eisenstein_eigenvalue}}{=} \lambda_x(\chi)\,E^{(n)}(\blanc,\chi) + n\,(\ln q_x)\,\lambda_x^-(\chi)\,E^{(n-1)}(\blanc,\chi) \ + \ \bigl(\text{lower terms}\bigr) \\
  & \hspace{3,7pt} =\hspace{3,7pt} \bigl(c_{n-1}\,\lambda_x(\chi) + n\,(\ln q_x)\,\lambda_x^-(\chi)\bigr)\,E^{(n-1)}(\blanc,\chi) \ + \ \bigl(\text{lower terms}\bigr) \;.
 \end{align*}
 By the induction hypothesis, $ \bigl\{ E(\blanc,\chi), E^{(1)}(\blanc,\chi),\ldots,E^{(n-1)}(\blanc,\chi)\bigr\}$ is linearly independent, and therefore
 $$ c_{n-1}\,\lambda_x(\chi) \ = \ c_{n-1}\,\lambda_x(\chi) + n\,(\ln q_x)\,\lambda_x^-(\chi) \;, $$
 which implies that $\lambda_x^-(\chi)=0$ for every place $x$. But this contradicts Lemma \ref{lemma_chi^2=0_and_lambda_x^-(chi)=0}.
\end{proof}

\begin{lemma}
 \label{linear_relation_between_E_and_E^1}
 Let $\chi\in\Xi_0$ such that $\chi^2=1$. Then
 $$ E^{(1)}(\blanc,\chi) \ = \ (\ln q)\ (2g-2)\ E(\blanc,\chi) \;. $$
\end{lemma}

\begin{proof}
 Since $\chi^2=1$, the functional equation looks like
 $$ E(g,\chi,s) \ = \ \norm{\fc}^{2s}\ E(g,\chi,-s) \;. $$
 Using $\norm{\fc}=q^{-(2g-2)}$ and taking derivatives in $s$ of both sides yields
 $$ E^{(1)}(g,\chi,s) \ = \ - \ \norm{\fc}^{2s}\ E^{(1)}(g,\chi,-s) \ + \ 2 \ (\ln q)\ (2g-2)\ \norm{\fc}^{2s}\ E(g,\chi,-s) \;, $$
 and filling in $s=0$ results in the desired equation.
\end{proof}

\begin{prop}
 \label{series_of_Eisenstein_derivatives_2-torsion}
 Let $\chi\in\Xi_0$ with $\chi^2=1$. Both
 $$ \bigl\{ E(\blanc,\chi),\ E^{(2)}(\blanc,\chi),\ E^{(4)}(\blanc,\chi),\ldots\bigr\} \quad \text{and} \quad
    \bigl\{ E^{(1)}(\blanc,\chi),\ E^{(3)}(\blanc,\chi),\ E^{(5)}(\blanc,\chi),\ldots\bigr\} $$ 
 span a vector space on which $\cH_K$ acts. If $g\neq1$, then both are linearly independent, but they span the same space. If $g=1$, then the former set is linearly independent and all functions in the latter set vanish.
\end{prop}

\begin{proof}
 That both sets span $\cH_K$-modules follows from Lemma \ref{gen_Eisenstein_eigenvalue} since by Lemma \ref{lemma_chi^2=0_and_lambda_x^-(chi)=0}, the value $\lambda_x^-(\chi)$ vanishes for all $x\in X$.

 The linear independence of the former set can be shown by the same calculation as in the proof of Proposition \ref{series_of_Eisenstein_derivatives}, provided one knows that $\lambda_x(\chi)\neq0$ for some $x\in X$. This holds since otherwise
 $$ 0 \ = \ \lambda_x(\chi)\, -\, \lambda_x^-(\chi) \ = \ 2q_x \, \chi(\pi_x) $$
 for all $x\in X$, which contradicts the nature of $\chi$.

 If $g\neq1$, then Lemma \ref{linear_relation_between_E_and_E^1} implies 
 that $E^{(1)}(\blanc,\chi)$ is a non-vanishing multiple of $E(\blanc,\chi)$ and
 spans thus the same vector space as $E(\blanc,\chi)$.
 Consequently the latter set in the Proposition is linearly independent for the same reasons as for the former set.
 Since $\cH_K$ is commutative, the two sets in question generate the same space.

 If $g=1$, the vanishing of all $E^{(i)}(\blanc,\chi)$ for odd $i$ 
 follows from the $i$-th derivative of the functional equation at $s=0$, which looks like
 $$ E^{(i)}(\blanc,\chi) \ = \ (-1)^i\, E^{(i)}(\blanc,\chi) \ + \ 
    \underbrace{(2g-2)}\limits_{=\ 0}\ (\text{terms in lower derivatives}) \;. \mbox{\qedhere} $$
\end{proof}

\begin{lemma}
 \label{gen_residue_eigenvalue}
 If $\chi\in\Xi_0$ with $\chi^2=\norm \ ^{\pm1}$, then 
 $$ \Phi_x R^{(i)}(g,\chi) \ = \ \sum_{k=0}^i \binom ik \bigl(\ln q_x\bigr)^{i-k} \, \lambda_x^{(i-k)}(\chi) \, R^{(k)}(g,\chi) $$
 for every $x\in X$, where $\lambda_x^{(l)}(\chi)$ are defined as in Lemma \ref{gen_Eisenstein_eigenvalue}.
\end{lemma}

\begin{proof}
 The proof is the same as for Lemma \ref{gen_Eisenstein_eigenvalue}. Note that the function $s\cdot E(\blanc,\chi)$ is holomorphic at $s=0$, so the limit in the definition of the residue and the limit in the definition of the derivative with regard to $s$ commute.
\end{proof}

\begin{prop}
 \label{series_of_residue_derivatives}
 Let $\chi\in\Xi_0$ with $\chi^2=\norm \ ^{\pm1}$. Then
 $$ \bigl\{ R(\blanc,\chi), R^{(1)}(\blanc,\chi),R^{(2)}(\blanc,\chi),\ldots\bigr\} $$
 is linearly independent and spans a vector space on which $\cH_K$ acts. In particular, none of these functions vanishes.
\end{prop}

\begin{proof}
 The proof is completely analogous to that of Proposition \ref{series_of_Eisenstein_derivatives}. Lemma \ref{lemma_chi^2=0_and_lambda_x^-(chi)=0} ensures us of the fact that $\lambda_x^-(\chi)\neq0$ for some $x\in X$.
\end{proof}

\begin{pg}
 \label{def_varE}
 We summarise the discussion as follows. For $\chi\in\Xi_0$, define
 $$ \varE^{(i)}(\blanc,\chi) \ = \ \left\{{\begin{array}{ll} E^{(i)}(\blanc,\chi) &\text{if }\chi^2\notin\{1,\norm \ ^{\pm1}\},\\
                                                  R^{(i)}(\blanc,\chi) &\text{if }\chi^2=\norm \ ^{\pm1},\\
						  E^{(2i)}(\blanc,\chi) &\text{if }\chi^2=1.\end{array}} \right.$$
 and $\varE(\blanc,\chi)=\varE^{(0)}(\blanc,\chi)$. Let $\varcE(\chi)^K\subset\varcE=\cE\oplus\cR$ be the space generated by all derivatives $\varE^{(i)}(\blanc,\chi)$ with $i\geq0$.

 Note that by the functional equations for Eisenstein series and their residues, the linear spaces spanned by the set
 $$ \bigl\{ \varE^{(0)}(\blanc,\chi),\ldots,\ \varE^{(n)}(\blanc,\chi)\bigr\} \quad\text{and}\quad
    \bigl\{ \varE^{(0)}(\blanc,\chi^{-1}),\ldots,\ \varE^{(n)}(\blanc,\chi^{-1})\bigr\} $$
 are the same for all $\chi\in\Xi_0$. In particular, $\varcE(\chi)^K=\varcE(\chi^{-1})^K$.
\end{pg}

\begin{thm}
 \label{main_thm_admissible_form} 
 The space $\cA^K$ of unramified automorphic forms decomposes as an $\cH_K$-module into 
 $$ \cA^K \ \ = \ \ \cA_0^K \ \ \oplus \!\!\bigoplus_{\{\chi,\chi^{-1}\}\,\subset\,\Xi_0}\!\! \varcE(\chi)^K \;. $$
 The vector space $\cA_0^K$ is finite-dimensional and admits a basis of $\cH_K$-eigenfunctions. For every $\chi\in\Xi_0$ and $n\geq0$, $\{\varE(\blanc,\chi),\dotsc\,\varE^{(n-1)}(\blanc,\chi)\}$ is a basis of the unique $\cH_K$-submodule of dimension $n$ in $\varcE(\chi)^K$. 
\end{thm}

\begin{proof}
 Note that all the spaces $\cA_0^K$ and $\varcE(\chi^K)$ are indeed $\cH_K$-modules, which follows for the latter spaces from Lemmas \ref{gen_Eisenstein_eigenvalue} and \ref{gen_residue_eigenvalue}. It is well-known that $\cA_0^K$ is finite-dimensional and the multiplicity one theorem implies that it has a basis of $\cH_K$-eigenfunctions (cf.\ \cite[Section 3.3]{Bump}). Since $\cP(\chi)\simeq\cP(\chi')$ only if $\chi'=\chi$ or $\chi'=\chi^{-1}$, there is no other linear relation of Eisenstein series than the one that is given by the functional equation. 
 
 Lemmas \ref{gen_Eisenstein_eigenvalue} and \ref{gen_residue_eigenvalue} imply that for every $\chi\in\Xi_0$, $\varcE(\chi)^K$ is an $\cH_K$-module. Propositions \ref{series_of_Eisenstein_derivatives}, \ref{series_of_Eisenstein_derivatives_2-torsion} and \ref{series_of_residue_derivatives} ensure that the described bases are indeed linearly independent.

 The uniqueness of the $n$-dimensional subspaces in the theorem follows from the Jordan decomposition theorem. We furthermore see that $\{\varE^{(i)}(\blanc,\chi)\}_{\substack{\{\chi,\chi^{-1}\}\subset\Xi_0},i\geq0}$ is linearly independent. Finally, it follows from Propositions \ref{series_of_Eisenstein_derivatives}, \ref{series_of_Eisenstein_derivatives_2-torsion} and \ref{series_of_residue_derivatives} together with the general remarks from paragraph \ref{thm_Li} that the decomposition exhausts $\cA^K$.
\end{proof}


\section{The dimension of the space of unramified toroidal Eisenstein series}
\label{section_dimension}

\begin{lemma}
 \label{even_mult}
 Let $\chi\in\Xi_0$ satisfy $\chi^2=1$. If $L(\chi,1/2)=0$, then $1/2$ is a zero of even multiplicity.
\end{lemma}

\begin{proof}
 Since the canonical divisor, which is represented by $\fc$, is a square in the divisor class group (cf.\ \cite[XIII.12, thm. 13]{Weil}) $\chi(\fc)=1$. Let $L^{(i)}(\chi,s)$ vanish at $s=1/2$ for all $i=0,\ldots,n-1$, for some odd $n$. We will show that in this case the multiplicity of $1/2$ as a zero must be strictly larger than $n$. Taking into account the vanishing of lower derivatives and $\chi(\fc)=1$, the $n$-th derivatives of both sides of the functional equation (cf.\ paragraph \ref{cont_L_series}) are
 $$ L^{(n)}(\chi,1/2) = (-1)^{n}\,L^{(n)}(\chi^{-1},1/2). $$
 Thus $L^{(n)}(\chi,1/2)=0$ as $(-1)^{n}=-1$ for odd $n$.
\end{proof}

\begin{pg}
 The $\cH_K$-module $\cE^K_\tor$ inherits a decomposition
 $$ \cE^K_\tor \ = \ \bigoplus_{\substack{\{\chi,\chi^{-1}\}\,\subset\,\Xi_0\\ \chi^2\neq\norm \ ^{\pm1}}}\!\! (\cE^K_\tor\,\cap\,\varcE(\chi)^K) $$
 from $\cA^K$ (see Theorem \ref{main_thm_admissible_form}). Only finitely many terms $\cE^K_\tor\,\cap\,\varcE(\chi)^K$ are nontrivial. Each of these terms has a basis of the form 
 $$\bigl\{ \varE(\blanc,\chi), \varE^{(1)}(\blanc,\chi),\ldots,\varE^{(n-1)}(\blanc,\chi)\bigr\} \;,$$
 where $n$ is its complex dimension.
 
 Thus it suffices to investigate Eisenstein series of the form $E(\blanc,\chi)$ and their derivatives $E^{(i)}(\blanc,\chi)$ for unramified quasi-characters $\chi$ in order to determine $\cE^K_\tor$. We will, however, state and prove theorems for general quasi-characters $\chi\in\Xi$ where no additional effort is required.
\end{pg} 

\begin{pg}
 \label{zero_of_L_gives_toroidal_Eisenstein_series}
 Let $\chi\in\Xi_0$ such that $\chi^2\neq\norm \ ^{\pm1}$. We say that $\chi$ is a zero of $L(\blanc,1/2)$ of order $n$ if $L(\chi,s+1/2)$ vanishes to order $n$ at $s=0$. By Theorem \ref{lemma_zagier_for_derivatives}, we see that if $\chi$ is a zero of $L(\blanc,1/2)$ of order $n$, then all the functions $E(\blanc,\chi),\dotsc,E^{(n-1)}(\blanc,\chi)$ are toroidal.

 The functional equation for $L$-series (paragraph \ref{cont_L_series}) implies that zeros come in pairs: $\chi$ is a zero of order $n$ if and only if $\chi^{-1}$ is a zero of order $n$, and if $\chi=\chi^{-1}$, then $\chi$ is a zero of even order (Lemma \ref{even_mult}). We call $\{\chi,\chi^{-1}\}$ a {\it pair of zeros of order $n$} if $\chi$ is a zero of order $n$ in case $\chi\neq\chi^{-1}$, or if $\chi$ is a zero of order $2n$ in case $\chi=\chi^{-1}$.

 Due to the definition of $\varcE(\chi)^K$ (in particular, notice the difference when $\chi^2=1$; see paragraph \ref{def_varE}) and because $\varcE(\chi)^K=\varcE(\chi^{-1})^K$, we obtain that if $\{\chi,\chi^{-1}\}$ is a pair of zeros of order $n$, then $\varE(\blanc,\chi),\dotsc,\varE^{(n-1)}(\blanc,\chi)$ are toroidal and span an $n$-dimensional $\cH_K$-module provided that $\chi^2\neq\norm \ ^{\pm1}$.

 We summarise this discussion.
\end{pg}

\begin{lemma}
 \label{zero_of_L_implies_toroidal}
 Let $\chi\in\Xi_0$ such that $\chi^2\neq\norm \ ^{\pm1}$ and $i\geq0$.
 \begin{enumerate}
  \item\label{tor1} Let $E/F$ be a separable quadratic algebra extension. Then $\varE^{(i)}(\blanc,\chi)$ is $E$-toroidal if and only if $\{\chi,\chi^{-1}\}$ is a zero of $L(\blanc,1/2)L(\blanc\chi_E,1/2)$ that is at least of order $i$.
  \item\label{tor2} If $\{\chi,\chi^{-1}\}$ is a pair of zeros of $L(\blanc,1/2)$ of order $n$, then $\varE(\blanc,\chi),\dotsc,\varE^{(n-1)}(\blanc,\chi)$ are toroidal.\qed
 \end{enumerate}
\end{lemma}

\begin{thm}
 \label{thm_dimension}
 If $q$ is odd, the dimension of $\cE^K_\tor$ equals $(g-1)h+1$, where $g$ is the genus and $h$ the class number of $F$. If $q$ is even, then $(g-1)h+1$ is a lower bound for the dimension of $\cE^K_\tor$.
\end{thm}

\begin{proof}
 By Theorem \ref{toroidal_if_and_only_L-series=0}, we have to consider only zeros of $L$-series $L(\chi,s)$ in order to know whether $E(\chi,s)$ resp.\ its derivatives are toroidal if $q$ is odd. For even $q$ there might be more toroidal derivatives of Eisenstein series, and the following only gives a lower bound for the dimension of $\cA_\tor^K$.

 Fix an idele $a_1\in\AA^\times$ of degree $1$ and let $\omega_1,\dots,\omega_{h}\in\Xi_0$ be the characters that are trivial on $\langle a_1\rangle$. Assume that $\omega_1$ is the trivial character. Then for every $\chi\in\Xi_0$, there is a unique $j\in\{1,\dotsc,h\}$ and $s\in\rquot{\CC}{(2\pi\i/\ln q)\ZZ}$ such that $\chi=\omega_j\norm \ ^s$. By class field theory, there is a finite abelian unramified extension $F'/F$ of order $h$ (cf.\ paragraph \ref{product_of_all_L_series}) such that 
 $$ \prod_{i=1}^{h}\ L_F(\omega_i,s+1/2) \ = \ \zeta_{F'}(s+1/2) \;. $$
 In particular the zeros of both hand sides as functions of $s$ are in one-to-one correspondence.

 By Weil's proof of the Riemann hypothesis for function fields, we know that this zeta function is of the form
 $$ \zeta_{F'}(s) = \frac{\fL_{F'}(q^{-s})}{(1-q^{-s})(1-q^{1-s})} $$
 for some polynomial $\fL_{F'}(T)\in\ZZ[T]$ of degree $2g_{F'}$ that has no zero at $T=1$ or $T=q^{-1}$ (cf.\ \cite[Thms. VII.4 and VII.6]{Weil}). This means that the orders of all pairs of zeros of $L(\blanc,1/2)$ sum up to $g_{F'}$, and that we find $g_{F'}$ linearly independent toroidal automorphic forms in $\cE^K$. Note that for a quasi-character $\chi=\omega_i\norm \ ^{s}$ with $\chi^2=\norm \ ^{\pm1}$, we have that $\zeta_{F'}(s+1/2)\neq0$ because $\fL(T)$ has no zero at $T=q^0$ or $T=q^{-1}$. Hence if $\chi$ is a zero of $L(\blanc,1/2)$, then $\varE(\blanc,\chi)$ is not a residuum.
  
 Finally, we apply Hurwitz' theorem (\cite[Cor.\ 2.4]{Hartshorne}) to the unramified extension $F'/F$ and obtain:
 $$ 2\,g_{F'} \, - \, 2 \ = \ h \, (\,2\,g \, - \, 2 \,)  \hspace{0,5cm}\text{and thus}\hspace{0,5cm} g_{F'} \ = \ (\,g \, - \, 1\,) \, h \, + \, 1 \;. \mbox{\qedhere} $$
\end{proof}

\begin{rem}
 Waldspurger calculated toroidal integrals of cusp forms over number fields. So assume for a moment that $F$ is a number field,
 $\pi$ an irreducible unramified cuspidal representation and $f\in\pi$ an unramified cusp form. 
 Let $L(\pi,s)$ be the $L$-function of $\pi$.
 Let $T\subset G$ be a torus corresponding to a separable quadratic field extension $E$ of $F$ and $\chi_T$ the character corresponding to $T$
 by class field theory. Then the square of the absolute value of
 $$ \int\limits_{\lquot{T_FZ_\AA}{T_\AA}} f(t)\,dt $$
 equals a harmless factor times $L(\pi,1/2)L(\pi\chi_T,1/2)$, cf.\ \cite[Prop.\ 7]{Waldspurger2}.
 
 These integrals are nowadays called Waldspurger periods of $f$, 
 and it is translated in some cases to global function fields, cf.\ \cite{Lysenko}. 
 This leads to the conjecture:
\end{rem}

\begin{conj}
 A cusp form $f$ of an irreducible unramified cuspidal subrepresentation $\pi$ of the space of automorphic forms 
 is toroidal if and only if $L(\pi,1/2)=0$.
\end{conj}

 By the multiplicity one theorem, this conjecture implies
 
\begin{conj}
 \label{conj_toroidal_cusp_forms}
 The dimension of $\cA_{0,\tor}^K$ equals the number of isomorphism classes of irreducible unramified cuspidal representations $\pi$
 with $L(\pi,1/2)=0$.
\end{conj}

\begin{rem}
 In \cite{Lorscheid3} one can find a proof of that $\cA_{0,\tor}^K=\{0\}$ if $g=1$ by a different method. Note that in this case, $L(\pi,s)$ has no zero for any irreducible unramified cuspidal representation.
\end{rem}

\newpage

\bibliographystyle{plain}

\begin{thebibliography}{10}

\bibitem{Bump}
Daniel Bump.
\newblock {\em Automorphic forms and representations}, volume~55 of {\em
  Cambridge Studies in Advanced Mathematics}.
\newblock Cambridge University Press, Cambridge, 1997.

\bibitem{Chinta-Friedberg-Hoffstein}
Gautam Chinta, Solomon Friedberg, and Jeffrey Hoffstein.
\newblock Multiple {D}irichlet series and automorphic forms.
\newblock In {\em Multiple {D}irichlet series, automorphic forms, and analytic
  number theory}, volume~75 of {\em Proc. Sympos. Pure Math.}, pages 3--41.
  Amer. Math. Soc., Providence, RI, 2006.

\bibitem{Clozel-Ullmo}
Laurent Clozel and Emmanuel Ullmo.
\newblock {\'E}quidistribution de mesures alg{\'e}briques.
\newblock {\em Compos. Math.}, 141(5):1255--1309, 2005.

\bibitem{Cornelissen-Lorscheid}
Gunther Cornelissen and Oliver Lorscheid.
\newblock Toroidal automorphic forms for some function fields.
\newblock {\em J. Number Theory}, 129(6):1456--1463, 2009.

\bibitem{Cornelissen-Lorscheid2}
Gunther Cornelissen and Oliver Lorscheid.
\newblock Toroidal automorphic forms, {W}aldspurger periods and double
  {D}irichlet series.
\newblock To appear in a proceedings of the Edinborough conference on
  $L$-series, {\tt arXiv:0906.5284}, 2009.

\bibitem{Drinfeld2}
Vladimir~G. Drinfel{\cprime}d.
\newblock Proof of the {P}etersson conjecture for {${\rm GL}(2)$} over a global
  field of characteristic {$p$}.
\newblock {\em Funktsional. Anal. i Prilozhen.}, 22(1):34--54, 96, 1988.

\bibitem{Fisher-Friedberg2}
Benji Fisher and Solomon Friedberg.
\newblock Double {D}irichlet series over function fields.
\newblock {\em Compos. Math.}, 140(3):613--630, 2004.

\bibitem{Friedberg-Hoffstein-Lieman}
Solomon Friedberg, Jeffrey Hoffstein, and Daniel Lieman.
\newblock Double {D}irichlet series and the {$n$}-th order twists of {H}ecke
  {$L$}-series.
\newblock {\em Math. Ann.}, 327(2):315--338, 2003.

\bibitem{Gelbart1}
Stephen~S. Gelbart.
\newblock {\em Automorphic forms on ad{\`e}le groups}.
\newblock Princeton University Press, Princeton, N.J., 1975.
\newblock Annals of Mathematics Studies, No. 83.

\bibitem{Gelbart-Jacquet}
Stephen~S. Gelbart and Herv{\'e} Jacquet.
\newblock Forms of {${\rm GL}(2)$} from the analytic point of view.
\newblock In {\em Automorphic forms, representations and $L$-functions, Part
  1}, Proc. Sympos. Pure Math., XXXIII, pages 213--251. Amer. Math. Soc.,
  Providence, R.I., 1979.

\bibitem{Hartshorne}
Robin Hartshorne.
\newblock {\em Algebraic geometry}.
\newblock Springer-Verlag, New York, 1977.
\newblock Graduate Texts in Mathematics, No. 52.

\bibitem{Hecke}
Erich Hecke.
\newblock {\"U}ber die {K}roneckersche {G}renzformel f{\"u}r reelle
  quadratische {K}{\"o}rper und die {K}lassenzahl relativ-abelscher
  {K}{\"o}rper.
\newblock In {\em Mathematische {W}erke}, pages 198--207. Vandenhoeck \&
  Ruprecht, G{\"o}ttingen, 1959.

\bibitem{Lachaud1}
Gilles Lachaud.
\newblock Z{\'e}ros des fonctions {$L$} et formes toriques.
\newblock {\em C. R. Math. Acad. Sci. Paris}, 335(3):219--222, 2002.

\bibitem{Lachaud2}
Gilles Lachaud.
\newblock Spectral analysis and the {R}iemann hypothesis.
\newblock {\em J. Comput. Appl. Math.}, 160(1-2):175--190, 2003.

\bibitem{Lang2}
Serge Lang.
\newblock {\em Algebra}, volume 211 of {\em Graduate Texts in Mathematics}.
\newblock Springer-Verlag, New York, third edition, 2002.

\bibitem{Li}
Wen Ch'ing~Winnie Li.
\newblock Eisenstein series and decomposition theory over function fields.
\newblock {\em Math. Ann.}, 240(2):115--139, 1979.

\bibitem{Lorscheid-thesis}
Oliver Lorscheid.
\newblock {\em Toroidal automorphic forms for function fields}.
\newblock PhD thesis, University of Utrecht, 2008.
\newblock available from {\tt http://igitur-archive.library.uu.nl}.

\bibitem{Lorscheid3}
Oliver Lorscheid.
\newblock Automorphic forms for elliptic function fields.
\newblock Preprint, 2010.

\bibitem{Lorscheid2}
Oliver Lorscheid.
\newblock Graphs of {H}ecke operators.
\newblock Preprint, 2010.

\bibitem{Lysenko}
Sergey Lysenko.
\newblock Geometric {W}aldspurger periods.
\newblock {\em Compos. Math.}, 144(2):377--438, 2008.

\bibitem{Moeglin-Waldspurger}
Colette M{\oe}glin and Jean-Loup Waldspurger.
\newblock {\em Spectral decomposition and {E}isenstein series}, volume 113 of
  {\em Cambridge Tracts in Mathematics}.
\newblock Cambridge University Press, Cambridge, 1995.

\bibitem{Titchmarsh}
Edward~Charles Titchmarsh.
\newblock {\em The theory of the {R}iemann zeta-function}.
\newblock The Clarendon Press Oxford University Press, New York, second
  edition, 1986.

\bibitem{Ulmer}
Douglas Ulmer.
\newblock Geometric non-vanishing.
\newblock {\em Invent. Math.}, 159(1):133--186, 2005.

\bibitem{Waldspurger4}
Jean-Loup Waldspurger.
\newblock Correspondance de {S}himura.
\newblock {\em J. Math. Pures Appl. (9)}, 59(1):1--132, 1980.

\bibitem{Waldspurger1}
Jean-Loup Waldspurger.
\newblock Quelques propri{\'e}t{\'e}s arithm{\'e}tiques de certaines formes
  automorphes sur {${\rm GL}(2)$}.
\newblock {\em Compositio Math.}, 54(2):121--171, 1985.

\bibitem{Waldspurger2}
Jean-Loup Waldspurger.
\newblock Sur les valeurs de certaines fonctions {$L$} automorphes en leur
  centre de sym{\'e}trie.
\newblock {\em Compositio Math.}, 54(2):173--242, 1985.

\bibitem{Waldspurger3}
Jean-Loup Waldspurger.
\newblock Correspondances de {S}himura et quaternions.
\newblock {\em Forum Math.}, 3(3):219--307, 1991.

\bibitem{Weil}
Andr{\'e} Weil.
\newblock {\em Basic number theory}.
\newblock Classics in Mathematics. Springer-Verlag, Berlin, 1995.

\bibitem{Wielonsky2}
Franck Wielonsky.
\newblock S{\'e}ries d'{E}isenstein, int{\'e}grales toro\"\i dales et une
  formule de {H}ecke.
\newblock {\em Enseign. Math. (2)}, 31(1-2):93--135, 1985.

\bibitem{Zagier1}
Don Zagier.
\newblock Eisenstein series and the {R}iemann zeta function.
\newblock In {\em Automorphic forms, representation theory and arithmetic
  (Bombay, 1979)}, volume~10 of {\em Tata Inst. Fund. Res. Studies in Math.},
  pages 275--301. Tata Inst. Fundamental Res., Bombay, 1981.

\end{thebibliography}

\end{document}